\documentclass[10pt]{amsart}

\usepackage{hyperref}

\usepackage{amssymb}
\usepackage{amsmath}
\usepackage{graphicx}
\usepackage{tikz}
\usepackage{appendix}
\usepackage{bibentry}
\usepackage{amsthm}
\usepackage{mathrsfs}

\let\newpf\proof \let\proof\relax

\newcommand{\bt}{\begin{thm}}
\newcommand{\et}{\end{thm}}

\newcommand{\bl}{\begin{lemma}}
\newcommand{\el}{\end{lemma}}

\newcommand{\beq}{\begin{eqnarray}}
\newcommand{\eeq}{\end{eqnarray}}

\def\be{\begin{equation}}
\def\ee{\end{equation}}

\def\ba{{\begin{align}}}
\def\ea{{\end{align}}}

\newtheorem{thm}{Theorem}[section]

\newtheorem{cor}{Corollary}[section]
\newtheorem{lemma}{Lemma}[section]

\newtheorem{sublemma}{}[lemma]

\newtheorem{prop}[thm]{Proposition}

\theoremstyle{remark}
\newtheorem{rem}{Remark}[section]

\numberwithin{equation}{section}

\def \bn {\hfill \\ \smallskip\noindent}

\theoremstyle{definition}

\def\proof{\bn {\bf Proof.} }

\newcommand{\C}{{\mathbb C}}

\newcommand{\N}{{\mathbb N}}
\newcommand{\Q}{{\mathbb Q}}
\newcommand{\R}{{\mathbb R}}
\newcommand{\T}{{\mathbb T}}

\newcommand{\Z}{{\mathbb Z}}

\begin{document}

\title{Quantum dynamical bounds for ergodic potentials with underlying dynamics of zero topological entropy}
\title[]{Quantum dynamical bounds for ergodic potentials with underlying dynamics of zero topological entropy}
\setcounter{tocdepth}{1}

\author{Rui Han and Svetlana Jitomirskaya}
\maketitle

\begin{abstract}
In this paper we obtain upper quantum dynamical bounds as a corollary of positive Lyapunov exponent for Schr\"odinger operators $H_{f,\theta} u(n)=u(n+1)+u(n-1)+ \phi(f^n\theta)u(n)$, where  $\phi : \mathcal{M}\to \R$ is a piecewise H\"older function on a compact Riemannian manifold $\mathcal{M}$, and $f:\mathcal{M}\to\mathcal{M}$ is a uniquely ergodic volume preserving map with zero topological entropy. As corollaries we obtain localization-type statements for shifts and skew-shifts on higher dimensional tori with arithmetic conditions on the parameters. These are the first localization-type results with precise arithmetic conditions for multi-frequency quasiperiodic and skew-shift potentials.
\end{abstract}

\section{introduction}
Positive Lyapunov exponents are generally viewed as a signature of localization. While it is known that they can coexist even with almost ballistic transport \cite{last} \cite{delrio}, vanishing of certain dynamical exponents has been identified as a reasonable expected consequence of hyperbolicity of the corresponding transfer-matrix cocycle. Results in this direction were obtained in \cite{DT1} \cite{DT2} for one-frequency trigonometric polynomials, and recently in \cite{JM1}, 
for one-frequency quasiperiodic potentials under very mild assumptions
on regularity of the sampling function. In this paper we identify a
general property responsible for positive Lyapunov exponents implying
vanishing of the dynamical quantitites in the rather general case of
underlying dynamics defined by volume preserving maps of Riemannian
manifolds with zero topological entropy, and under very minimal
regularity assumptions. This work presents the first localization-type
results that hold in such generality. We expect that positive
topological entropy should also lead to vanishing of the dynamical
quantities for a.e. (but not every!) phase, but this should be approached by completely different methods and will be explored in a future work. 

Our general results allow us, in particular, to obtain localization-type statements for potentials defined by shifts and skew-shifts of higher-dimensional tori. Pure point spectrum with exponentially decaying eigenfunctions has been obtained for a.e. multi-frequency shifts  in the regime of positive Lyapunov exponents in \cite{bg} and for the skew-shift on $\T^2$ with a perturbative condition in \cite{bgs}, both very delicate results. While bounds on transport exponents are certainly weaker than dynamical localization that often (albeit not always \cite{jss}) accompanies pure point spectrum \cite{bj}, we note that pure point spectrum can be destroyed by generic rank one perturbations \cite{dms} while vanishing of the transport exponents is robust in this respect. Finally, our results are the first ones for both of these families that hold under purely arithmetic conditions and the first non-perturbative ones for the skew-shift.

Let $(\mathcal{M}, g)$ be a $d$-dimensional compact (smooth) Riemannian manifold with a metric $g$. 
Let $\mathrm{Vol}_g$ be its Riemannian volume density (see (\ref{Volg})). 
Let $f$ be a uniquely ergodic volume preserving map on $\mathcal{M}$, which means $\mathrm{Vol}_g$ is its unique invariant probability measure. 
We will study the dynamical properties of the Schr\"odinger operator acting on $l^2(\Z)$:
\begin{align}\label{schrodingeropt}
H_{f,\theta} u(n)=u(n+1)+u(n-1)+ \phi(f^n\theta)u(n).
\end{align} 
where $\theta\in \mathcal{M}$ is the phase.

The time dependent Schr\"odinger equation
\begin{align*}
i \partial_t u=H_{\theta}u,
\end{align*}
leads to a unitary dynamical evolution
\begin{align*}
u(t)=e^{-itH_{\theta}}u(0).
\end{align*}
Under the time evolution, the wavepacket will in general spread out with time.
For operators with absolutely continuous spectrum, scattering theory leads to a good understanding of the quantum dynamics.
In this paper we will study the spreading of the wavepacket under positive Lyapunov exponent assumption, which automatically implies the absence of absolutely continuous spectrum.

Let $e^{-itH_{\theta}}\delta_0$ be the time evolution with the localized initial state $\delta_0$. Let
\begin{align*}
a_{\theta}(n,t)=|\langle e^{-itH_{\theta}}\delta_0, \delta_n \rangle|^2.
\end{align*}
$a_{\theta}(n,t)$ describes the probability of finding the wavepacket at site $n$ at time $t$. 
We denote the $p$-th moment of $a_{\theta}(n,t)$ by
\begin{align*}
\langle |X|_{\theta}^p(t)\rangle=\sum_n (1+|n|)^p a_{\theta}(n,t).
\end{align*}

Dynamical localization is defined as boundedness of $\langle |X|_{\theta}^p(t)\rangle$ in time $t$.
This implies purely point spectrum, therefore for general operators with positive Lyapunov exponent  such a strong control of the wavepacket is not possible.
Thus we need to define proper transport exponents which decribe the rate of the spreading of the wavepacket.
For $p>0$  define the upper and lower transport exponents
\begin{align*}
\beta_{\theta}^+(p)=\limsup_{t\rightarrow\infty}\frac{\ln{\langle |X|_{\theta}^p(t)\rangle}}{p\ln{t}};\ \ \beta_{\theta}^-(p)=\liminf_{t\rightarrow\infty}\frac{\ln{\langle |X|^p_{\theta}(t)\rangle}}{p\ln{t}}.
\end{align*}
Obtaining upper bounds for the two transport exponents above implies a power-law control of the spreading rate of the entire wavepacket.

It is also interesting to consider a portion of the wavepacket. 
For a nonnegative function $A(t)$ of time, let
\begin{align*}
\langle A(t)\rangle_T=\frac{2}{T}\int_0^\infty e^{-2t/T}A(t)\ dt
\end{align*}
be its time average.
Set
\begin{align*}
P_{\theta,T}(L)=\sum_{|n|\leq L}\langle a_{\theta}(n,t)\rangle_T.
\end{align*}
Roughly speaking, $P_{\theta, T}(T^a)>\tau$ means that, in average, over time $T$, a portion of the wavepacket stays inside a box of size $T^a$.
Let us introduce two other scaling exponents:
\begin{align*}
\overline{\xi_\theta}&=\lim_{\tau\rightarrow 0}\limsup_{T\rightarrow \infty}\frac{\ln{\inf\{L|P_{\theta,T}(L)+P_{f\theta,T}(L)>\tau\}}}{\ln{T}}\\
\underline{\xi_\theta}&=\lim_{\tau\rightarrow 0}\liminf_{T\rightarrow \infty}\frac{\ln{\inf\{L|P_{\theta,T}(L)+P_{f\theta, T}(L)>\tau\}}}{\ln{T}}
\end{align*}

The vanishing of $\beta^{\pm}$ and $\overline{\xi}$, $\underline{\xi}$ can be viewed as localization-type statements.
For $\mathcal{M}=\T$ the one-dimensional torus, $f: \theta\rightarrow \theta+\alpha$ the irrational rotation, the Lebesgue measure $m$ is the unique invariant probability measure of $f$. 
It was first proved in \cite{DT1}, \cite{DT2} that in this case for $\phi$ being a trigonometric polynomial, under the assumption of positive Lyapunov exponent, $\beta_{\theta}^{+}(p)=0$ for all $p>0$, all $\theta$ and Diophantine $\alpha$; 
$\beta_{\theta}^{-}=0$ for all $p>0$, all $\theta$ and all $\alpha$. 
It was recently proved in \cite{JM1} that under very mild restrictions on regularity of the potential, 
under the assumption of positivity and continuity of the Lyapunov exponent, 
$\beta_{\theta}^{+}(p)=0$ for all $p>0$, all $\theta$ and Diophantine $\alpha$; $\beta_{\theta}^{-}(p)=0$ for all $p>0$, all $\theta$ and all $\alpha$. 
It was also proved in \cite{JM1} that for piecewise H\"older function, under the assumption of positive Lyapunov exponent, $\overline{\xi}_{\theta}=0$ for a.e.$\theta$ and Diophantine $\alpha$, $\underline{\xi}_{\theta}=0$ for a.e.$\theta$ and all $\alpha$.

\begin{rem}
The two Diophantine sets of $\alpha$ are different between \cite{DT1}, \cite{DT2} and \cite{JM1}. They are both full measure sets, but \cite{JM1} covers a slightly thinner set of frequencies because of the need to handle potentials with weaker regularity.
\end{rem}

In this paper we consider $d$-dimensional compact Riemannian manifold $\mathcal{M}$ and uniquely ergodic volume preserving map $f$. We consider maps with the following volume scaling property. 
For $1\leq l\leq d$, let $\Sigma (l)$ be the set of $C^{\infty}$ mappings $\sigma: Q^l \rightarrow \mathcal{M}$ where $Q^l$ is the $l$-dimensional unit cube. 
Let $\mathrm{Vol}_{g,l}(\sigma)$ be the induced $l$-dimensional volume of the image of $\sigma$ in $\mathcal{M}$ counted with multiplicity, i.e. 
if $\sigma$ is not one-to-one, and the image of one part coincides with that from another part, then we will count the set as many times as it is covered. 
For $n=1,2,...$ and $1\leq l\leq d$, let
\begin{align}\label{scale}
V_{l}(f)=\sup_{\sigma\in \Sigma (l)}\limsup_{n\rightarrow \infty}\frac{1}{n} \log{\mathrm{Vol}_{g, l}(f^n \sigma)}\ \ \ \mathrm{and}\ \ \ V(f)=\max_{l}V_{l}(f).
\end{align}
Volume preserving $f$ always satisfies $V_d(f)=V_d(f^{-1})=0$. 
Here we need to make an extra assumption that $V(f)=V(f^{-1})=0$. 
It is known that for smooth invertible map $f$, $V(f)=V(f^{-1})$ is equal to the {\it topological entropy} of $f$ \cite{Yod}, thus our class of maps includes all smooth maps with zero topological entropy.
In particular, it includes both the irrational rotation and the skew-shift.

For such maps we will assume that $f$ has a bounded discrepancy.

Let $J_N(\theta)=J(\theta, f\theta,..., f^{N-1}\theta)$ (see (\ref{defisodiscre})) be the isotropic discrepancy function of the sequence $\{f^n\theta \}_{n=0}^{N-1}$. 
For $\delta>0$, we will say $f$ has {\it strongly $\delta$-bounded isotropic discrepancy} if $J_N(\theta)\leq |N|^{-\delta}$ uniformly in $\theta $ for $|N|>N_0$; 
$f$ has {\it weakly $\delta$-bounded isotropic discrepancy} if there exists a sequence $\{N_j\}$ such that $J_{N_j}(\theta)\leq |N_j|^{-\delta}$ uniformly in $\theta$.
It turns out many concrete dynamical systems feature these properties. 
We will show in Lemmas \ref{toralsbdd} - \ref{skewwbdd} that the following holds.
\begin{itemize} 
\item Shifts of higher dimensional tori, $f:\theta \rightarrow \theta+\alpha$, has strongly bounded isotropic discrepancy for Diophantine $\alpha$;
\item Skew-shift $f:(y_1, y_2,... ,y_d)\rightarrow (y_1+\alpha, y_2+y_1,... ,y_d+y_{d-1})$, has strongly bounded isotropic discrepancy for Diophantine $\alpha$,  and weakly bounded isotropic discrepancy for Liouvillean $\alpha$.
\end{itemize}


Under the assumption of boundedness of discrepancy and scaling property of  $f$, we are ready to formulate the following two abstract results.

Let $\mu_{\theta}$ be the spectral measure of $H_{\theta}$ corresponding to $\delta_0$. Let $N=\int_{\mathcal{M}} \mu_{\theta}\ \mathrm{d}\mathrm{Vol}_g$ be the integrated density of states. Let $L(E)$ be the Lyapunov exponent, see (\ref{lyap}).

\begin{thm}\label{xi}
Let $\phi$ be a piecewise H\"older function. Suppose $L(E)$ is positive on a Borel subset $U$ with $N(U)>0$. 
Suppose $f$ is a uniquely ergodic volume preserving map satisfying $V(f)=V(f^{-1})=0$. We have
\begin{itemize}
\item If for some $\delta>0$, $f$ has weakly $\delta$-bounded isotropic discrepancy, then $\underline{\xi_{\theta}}=0$ for $\mathrm{Vol}_g$-a.e. $\theta\in \mathcal{M}$;
\item If for some $\delta>0$, $f$ has strongly $\delta$-bounded isotropic discrepancy, then $\overline{\xi_{\theta}}=0$ for $\mathrm{Vol}_g$-a.e. $\theta\in \mathcal{M}$.
\end{itemize}
\end{thm}
\begin{rem}
The full measure set of $\theta$ appearing in Theorem $\ref{xi}$ is precisely the set $\{\theta: \mu_\theta+\mu_{f\theta}(U)>0\}$.
\end{rem}

\begin{thm}\label{beta}
Under the assumption of Theorem \ref{xi}, assume also $L(E)$ is continuous in $E$ and $L(E)>0$ for every $E\in \R$. We have
\begin{itemize}
\item If for some $\delta>0$, $f$ has weakly $\delta$-bounded isotropic discrepancy, then $\beta^-_{\theta}(p)=0$ for all $\theta\in \mathcal{M}$ and $p>0$;
\item If for some $\delta>0$, $f$ has strongly $\delta$-bounded isotropic discrepancy, then $\beta^+_{\theta}(p)=0$ for all $\theta\in \mathcal{M}$ and $p>0$.
\end{itemize}
\end{thm}

\begin{rem} Strongly $\delta$-bounded isotropic discrepancy is essential for vanishing of $\overline{\xi}$ and $\beta^+_{\theta}(p)$, see Remarks \ref{xirem} and \ref{betarem}. 
However, it is not yet clear whether weakly $\delta$-bounded isotropic discrepancy (or any condition at all other than mere positivity of the Lyapunov exponent) is essential for vanishing of the $\underline{\xi}$ or of $\beta^-_{\theta}.$ 
\end{rem}
Theorems \ref{xi}, \ref{beta} extend the results of \cite{DT1,DT2,JM1} from
irrational rotations of the circle to general uniquely ergodic maps of
compact Riemannian manifolds with  zero
topologogical entropy and bounded discrepancy. One key to achieving
such generality is a new argument that does not rely on harmonic
analysis/ approximation by trigonometric polynomials.

By \cite{DT3}, $\beta_{\theta}^{-}(p)\geq p\dim_{H}(\mu_{\theta})$ where $\dim_{H}(\mu)$ is the Hausdorff dimension of $\mu$. Thus as a consequence of $\beta^{-}_{\theta}(p)=0$ we have the following
\begin{cor}
Under the assumption of Theorem \ref{beta}, $\dim_{H}(\mu_{\theta})=0$ for all $\theta\in \mathcal{M}$.
\end{cor}
\begin{rem} The point here is that we obtain zero Hausdorff dimension of the spectral measure for {\it all} rather than a.e. $\theta\in \mathcal{M}$ (the latter is known for general ergodic potentials \cite{simon}). The statement for all $\theta$ has only been known for irrational rotations of $\T^1$ (proved for trigonometric polynomials in \cite{JL}, and follows easily for piecewise functions from the results of \cite{JM1}).

The following Theorems \ref{xitoral} - \ref{betaskew} are all corollaries of our abstract results. 
Theorems \ref{xi2d} and \ref{beta2d} depend on a somewhat different
technique (bypassing the discrepancy considerations), which allows us
to cover more frequencies in case of the shift of $\T^2.$
To our knowledge, Theorems \ref{xitoral} -\ref{beta2d} are the first arithmetic localization-type results.
\end{rem}

Let us introduce the Diophantine condition and weak Diophantine condition on $\T^d$:
\begin{align*}
DC(\tau)=\cup_{c>0}DC(c,\tau)=\cup_{c>0}\{(\alpha_1,..., \alpha_d)| \|\langle \vec{h}, \alpha \rangle\|_{\R/{\Z}} \geq \frac{c}{r(\vec{h})^{\tau}} \ \mathrm{for}\ \mathrm{any}\ \vec{h}\neq \vec{0}\}
\end{align*}
where $r(\vec{h})=\prod_{i=1}^d \max{(|h_i|,1)}$. 
It is well known that when $\tau>1$, $DC(\tau)$ is a full measure set.
\begin{align*}
WDC(\tau)=\cup_{c>0}WDC(c, \tau)=\cup_{c>0}\{(\alpha_1,...,\alpha_d)| \max\{\|h\alpha_i\|_{\R/\Z}\}\geq \frac{c}{|h|^{\tau}}\ \mathrm{for}\ \mathrm{any}\ h\neq 0\}, h \in \Z.
\end{align*}
It is well known that when $\tau>\frac{1}{d}$, $WDC(\tau)$ is a full measure set.

Theorem \ref{xi} reduces vanishing of (upper or lower) $\xi_\theta$ to bounds on the
isotropic discrepancy. As corollaries, we obtain

\begin{thm}\label{xitoral}
Let $f$ be an irrational shift on $\T^d$. For piecewise H\"older $\phi$, suppose $L(E)$ is positive on a Borel subset $U$ with $N(U)>0$. 
Then if $\alpha\in DC(\tau)\subset \T^d$, $\tau>1$, we have $\overline{\xi_{\theta}}=0$ for $a.e.\ \theta \in \T^d$.
\end{thm}
\begin{rem}\label{xirem}
The Diophantine condition is essential for the vanishing of $\overline{\xi}$ \cite{JZ}.
\end{rem}

\begin{thm}\label{xiskew}
Let $f$ be a skew-shift. For piecewise H\"older $\phi$, suppose $L(E)$ is positive on a Borel subset $U$ with $N(U)>0$. Then
\begin{itemize}
\item for all irrational $\alpha$, $\underline{\xi_{\vec{y}}}=0$ for a.e. $\vec{y}\in \T^d$,
\item if $\alpha\in DC(\tau)$ for some $\tau>1$, $\overline{\xi_{\vec{y}}}=0$ for a.e. $\vec{y}\in \T^d$.
\end{itemize}
\end{thm}

\begin{rem}
The full measure set appearing in Theorems \ref{xitoral} and \ref{xiskew} is precisely the set $\{\theta: \mu_\theta+\mu_{f\theta}(U)>0\}$.
\end{rem}

Similarly, for systems with continuous Lyapunov exponent, Theorem \ref{beta} reduces vanishing of
$\beta^{\pm}_{\theta}(p)$ to the same discrepancy bounds, and we
obtain

\begin{thm}\label{betatoral}
Under the assumption of Theorem \ref{xitoral}, assume in addition that $L(E)$ is continuous in $E$ and $L(E)>0$ for every $E\in \R$. 
Then if $\alpha\in DC(\tau)\subset \T^d$, $\beta^+_{\theta}(p)=0$ for all $\theta \in \T^d$, $p>0$.
\end{thm}
\begin{cor}
Under the assumption of Theorem \ref{betatoral}, if $\alpha\in DC(\tau)$, $\dim_{H}(\mu_{\theta})=0$ for all $\theta\in \T^d$.
\end{cor}
\begin{rem}\label{betarem}
The Diophantine condition is essential for $\beta^{+}=0$ \cite{JZ}.
\end{rem}

\begin{thm}\label{betaskew}
Under the assumption of Theorem \ref{xiskew}, assume in addition that $L(E)$ is continuous in $E$ and $L(E)>0$ for every $E\in \R$. Then
\begin{itemize}
\item for all irrational $\alpha$, $\beta^-_{\vec{y}}(p)=0$ for all $\vec{y}\in \T^d$, $p>0$,
\item if $\alpha\in DC(\tau)$ for some $\tau>1$, $\beta^+_{\vec{y}}(p)=0$ for all $\vec{y}\in \T^d$, $p>0$.
\end{itemize}
\end{thm}
\begin{cor}
Under the assumption of Theorem \ref{betaskew}, for all irrational $\alpha$, $\dim_{H}(\mu_{\vec{y}})=0$ for all $\vec{y}\in \T^d$.
\end{cor}

Finally, for the case of the irrational shift $\T^2$ we can make two
more delicate statements, using a different technique to
obtain arithmetic estimates.

\begin{thm}\label{xi2d}
Let $f$ be an irrational shift on $\T^2$. For piecewise H\"older $\phi$, suppose $L(E)$ is positive on a Borel subset $U$ with $N(U)>0$. Then if $\alpha=(\alpha_1,\alpha_2)\in \cup_{\tau>1} WDC(\tau)$, we have
$\underline{\xi_{\theta}}=0$ for $a.e.\ \theta \in \T^2$.
\end{thm}
\begin{rem}
The full measure set appearing in Theorem \ref{xi2d} is precisely the set $\{\theta: \mu_\theta+\mu_{f\theta}(U)>0\}$.
\end{rem}

\begin{thm}\label{beta2d}
Under the assumption of Theorem \ref{xi2d}, assume in addition that $L(E)$ is continuous in $E$ and $L(E)>0$ for every $E\in \R$. Then if $\alpha=(\alpha_1, \alpha_2)\in \cup_{\tau>1}WDC(\tau)$, we have $\beta^{-}_{\theta}(p)=0$ for all $\theta\in \T^2$, $p>0$.
\end{thm}
\begin{cor}
Under the assumption of Theorem \ref{beta2d}, if $\alpha\in \cup_{\tau>1}WDC(\tau)$, we have $\dim_{H}(\mu_{\theta})=0$ for all $\theta\in \T^2$.
\end{cor}

The most technically complex part of the paper consists in obtaining
arithmetic estimates on covering of the torus by the
trajectory of a small ball in a polynomial (in the inverse radius)
time, which we obtain by estimating the discrepancy in
Theorems \ref{xitoral} - \ref{betaskew}, and by the bounded remainder
set technique in Theorems \ref{xi2d}, \ref{beta2d}. The discrepancy estimates are standard for the Diophantine
shifts and are ideologically similar to the known results on
equidistribution of $n^k\alpha$, for the case of higher dimensional
Diophantine skew shifts. We still develop the proof for the
Diophantine skew shift case in full detail because we did not find it in the
literature and also because it serves as a good preparation to the
Liouville higher dimensional skew shift, for which to the best of our knowledge, our
estimates are new. We note that for the {\it Diophantine} skew shift
of $T^2$ and shifts of
$T^d$ the results on the covering of the torus by a trajectory of a
ball are shown in \cite{adz} by a completely different technique,
through solving the cohomological equation. By the nature of the
cohomological equation that technique is not
extendable to the Liouville or weakly Diophantine case.

We organize this paper as follows: in section 2 we introduce some basic definitions. Some of them have been mentioned in the introduction but not in detail. In section 3 we will present some key lemmas and prove Theorems $\ref{xi}$ - $\ref{beta2d}$. In sections 4-8 we prove the key lemmas that are listed in section 3.

\section{Preparation}
\subsection{Riemannian manifolds.}
Let $\mathcal{M}$ be a $d$-dimensional compact Riemannian manifold with a Riemannian metric $g$.

Let $K$ be a compact set in some coordinate patch $(U, x^1,..., x^d)$. 
We define the volume of $K$ to be
\begin{align*}
\mathrm{Vol}_g(K):=\int_{x(K)}\sqrt{|G \circ x^{-1}|} dx^1\cdots dx^d,
\end{align*}
where $G=\det{g_{ij}}$, $g_{ij}=g(\frac{\partial}{\partial x_i}, \frac{\partial}{\partial x_j})$ and $dx^1\cdots dx^d$ is the Lebesgue measure on $\R^d$.
This definition is free of choice of coordinate.
If $K$ is not contained in a single coordinate patch, one could apply partition of unity to define $\mathrm{Vol}_g(K)$.
More precisely, we pick an atlas $(U_{\alpha}, x_\alpha^1,..., x^d_{\alpha})$ of $\mathcal{M}$ and a partition of unity $\{\rho_{\alpha}\}$ subordinate to this atlas. 
Now we can set
\begin{align*}
\mathrm{Vol}_g(K)=\sum_{\alpha}\int_{x^{\alpha}(K \cap U_{\alpha})}(\rho_{\alpha}\sqrt{|G^{\alpha}}|)\circ (x^{\alpha})^{-1} dx_{\alpha}^1\cdots dx_{\alpha}^d.
\end{align*}
The {\it Riemannian volume density} (see e.g.\cite{Nicolaescu}, section 3.4) on $(\mathcal{M}, g)$ is
\begin{align}\label{Volg}
\mathrm{d}\mathrm{Vol}_g=\sum_{\alpha}(\rho_{\alpha}\sqrt{|G^{\alpha}|}) \circ (x^{\alpha})^{-1} dx_{\alpha}^1\cdots dx_{\alpha}^d.
\end{align}
In the above definition, we do not assume $\mathcal{M}$ to be oriented. 
If $\mathcal{M}$ is oriented, then the volume density is actually a positive $n$-form, called the volume form.

If $\varrho: [a,b]\rightarrow \mathcal{M}$ is a continuously differentiable curve in the Riemannian manifold $\mathcal{M}$, then we define its length $l(\varrho)$ by
\begin{align*}
l(\varrho)=\int_{a}^{b}\sqrt{g_{\varrho(t)}(\dot{\varrho} (t), \dot{\varrho} (t))}\ dt,
\end{align*}
where $g_{\varrho(t)}$ is the inner product $g$ at the point $\varrho(t)$.
One could define the distance between any two point $x$, $y\in \mathcal{M}$ as follows
\begin{align*}
&dist(x, y)\\
=&\inf\{l(\varrho): \varrho\ \mathrm{is}\ \mathrm{a}\ \mathrm{continuous},\ \mathrm{piecewise}\ \mathrm{continuously}\ \mathrm{differentiable}\ \mathrm{curve}\ \mathrm{connecting}\ x\ \mathrm{and}\ y\}.
\end{align*}
With the definition of distance, {\it geodesics} in a Riemannian manifold are then the locally distance-minimizing paths.

Let $v\in\mathrm{T}_x\mathcal{M}$ be a tangent vector to the manifold $\mathcal{M}$ at $x$. 
Then there is a unique geodesic $\varrho_v$ satisfying $\varrho_v(0)=x$ with initial tangent vector $\dot{\varrho}_v(0)=v$. 
The corresponding {\it exponential map} is defined by $\exp_x(v)=\varrho_v(1)$.

Let $B_{r}(x)=\{y\in \mathcal{M}: dist(x, y)<r\}$ be a {\it geodesic ball} centered at $x\in \mathcal{M}$ with radius $r$. 
It is known that $B_r(x)=\exp_{x}(B(0, r))$ where $B(0, r)=\{v\in \mathrm{T}_x\mathcal{M}: g_{x}(v, v)<r\}$.

\begin{prop}\label{volumegeoball}
There exists $r_{g}>0$ so that for all $r<r_{g}$, there exist positive constants $C_g$ and $c_g$ which are independent of $x\in \mathcal{M}$ so that
\begin{align}
 c_g r^d\leq \mathrm{Vol}_g(B_r(x)) \leq C_g r^d\ \mathrm{for}\ \mathrm{any}\ x\in \mathcal{M}.
\end{align}
\end{prop}
\begin{proof}
We will discuss the proof briefly.
We could identify the tangent space $\mathrm{T}_x\mathcal{M}$ isometrically with $\R^d$. 
Now $\exp_x:\R^d \rightarrow \mathcal{M}$ is a diffeomorphism on some small ball $B_{\R^d}(0, r)$. 
On this ball, straight lines are mapped to length-minimizing geodesics (\cite{Carmo}, Proposition 3.6), 
and thus Euclidean balls are mapped to geodesic balls of the same radius. 
Taking $r$ smaller if necessary, we can assume the Jacobian of $\exp_x$ is bounded away from $0$ and $\infty$ on $B_{\R^d}(0,r)$, thus for $r<r_x$ we have that $c_{g_x} r^d \leq \mathrm{Vol}_g⁡(B_r(x))\leq C_{g_x} r^d$.
Since $\mathcal{M}$ is a compact manifold, we could take $r_x, c_{g_x}, C_{g_x}$ independent of $x\in \mathcal{M}$. $\hfill{} \Box$
\end{proof}


A subset $C$ of $\mathcal{M}$ is said to be a {\it geodesically convex set} if, given any two points in $C$, there is a minimizing geodesic contained within $C$ that joins those two points.

The {\it convexity radius at a point} $x\in\mathcal{M}$ is the supremum (which may be $+\infty$) of $r_x\in \R$ such that for all $r<r_x$ the geodesic ball $B_{r_x}(x)$ is geodesically convex. The {\it convexity radius of $(\mathcal{M}, g)$} is the infimum over the points $x \in \mathcal{M}$ of the convexity radii at these points.

\begin{prop}\cite{Berger}\label{ballconvex}
For compact manifold $\mathcal{M}$, the convexity radius $r_{g}^\prime$ of $(\mathcal{M}, g)$ is positive. 
\end{prop}
This clearly implies that for any $x\in \mathcal{M}$, any $r<r_{g}^\prime$, $B_{r}(x)$ is geodesically convex.

\subsection{Piecewise H\"older functions}\label{piecewiseholderf}
Let $L_{\gamma}(\mathcal{M})$ be the space of $\gamma$-Lipschitz functions on $\mathcal{M}$. 
For $\phi\in L_{\gamma}(\mathcal{M})$ define
\begin{align}\label{Lip}
\|\phi\|_{L_{\gamma}}=\|\phi\|_{\infty}+\sup_{\theta_1, \theta_2\in M}\frac{|\phi(\theta_1)-\phi(\theta_2)|}{dist{(\theta_1, \theta_2)}^{\gamma}}.
\end{align}
We say $\phi$ is piecewise H\"older if there exists $\gamma>0$, positive integer $K$ and $\{\phi_{j}\}_{j=1}^K\subset L_{\gamma}(\mathcal{M})$ so that
\begin{align*}
\phi(\theta)=\sum_{j=1}^{K}\chi_{S_j}(\theta)\phi_j(\theta)
\end{align*}
where $\{S_j\}_{j=1}^M$ are sets with ``good boundary'', namely $\{\partial{S_j}\}_{j=1}^K$ are $d-1$ dimensional smooth submanifolds of $\mathcal{M}$.
Clearly the discontinuity set $J_{\phi}$ of $\phi$ is $\cup_{j=1}^K \partial{S_j}$, and
\begin{align}\label{d-1measurepartial}
\mathrm{Vol}_{g, d-1}(J_{\phi})\leq \sum_{j=1}^K \mathrm{Vol}_{g, d-1}(\partial{S_j})<\infty.
\end{align}
Clearly for any two points $\theta_1, \theta_2$ so that $dist(\theta_i, J_{\phi})\geq r$, if $dist(\theta_1, \theta_2)<r$ then we have
\begin{align}\label{differencenorm}
|\phi(\theta_1)-\phi(\theta_2)|\leq dist(\theta_1, \theta_2)^{\gamma} \sum_{j=1}^K \|\phi_j\|_{L_{\gamma}}.
\end{align}

\subsection{Cocycles and Lyapunov exponent}
We now introduce the Lyapunov exponent. For a given $z\in\C$, a formal solution $u$ of $Hu=zu$  can be reconstructed using the transfer matrix
\begin{align*}
A(\theta,z)=
\left(
\begin{matrix}
z-\phi(\theta)\ \ &-1\\
1\ \ &0
\end{matrix}
\right)
\end{align*}
via the equation
\begin{align*}
\left(
\begin{matrix}
u(n+1)\\
u(n)
\end{matrix}
\right)
=
A(f^n \theta,z)
\left(
\begin{matrix}
u(n)\\
u(n-1)
\end{matrix}
\right)
\end{align*}

Indeed, let $A_k(\theta,z)$ be the product of consecutive transfer matrices:
\begin{align*}
&A_k(\theta,z)=A(f^{k-1}\theta,z)\cdot\cdot\cdot A(f\theta, z)A(\theta,z)\ \ \mathrm{for}\ k>0,\ \ \ \ A_0(\theta,z)=I\ \ \mathrm{and}\\
&A_k(\theta,z)=(A_{-k}(f^k \theta,z))^{-1}\ \ \mathrm{for}\ k<0.
\end{align*}
Then for any $k\in\Z$ we have the following relation
\begin{align*}
\left(
\begin{matrix}
u(k)\\
u(k-1)
\end{matrix}
\right)
=
A_k(\theta,z)
\left(
\begin{matrix}
u(0)\\
u(-1)
\end{matrix}
\right).
\end{align*}

We define the Lyapunov exponent
\begin{align}\label{lyap}
L(z)=\lim_k \frac{1}{k}\int_{\mathcal{M}} \ln{\|A_k(\theta,z)\|}\ \mathrm{d} \mathrm{Vol}_g(\theta)=\inf_{k} \frac{1}{k}\int_{\mathcal{M}} \ln{\|A_k(\theta,z)\|}\ \mathrm{d} \mathrm{Vol}_g(\theta).
\end{align}
Furthermore, $L(z)=\lim_k \frac{1}{k} \ln{\|A_k(\theta,z)\|}$ for $\mathrm{Vol}_g$-$a.e.\ \theta\in \mathcal{M}$.

\subsection{Spectral measure and integrated density of states}
Let $\mu_{\theta}$ be the spectral measure of $H_{\theta}$ corresponding to $\delta_0$ defined by
\begin{align*}
\langle (H_{\theta}-z)^{-1}\delta_0,\delta_0\rangle=\int_{\R} \frac{d\mu_{\theta}(x)}{x-z}.
\end{align*}
Then clearly $\mu_{f \theta}$ is the spectral measure of $H_{\theta}$ corresponding to $\delta_1$. Let $N=\int_{\mathcal{M}}\mu_{\theta}\ \mathrm{d}\mathrm{Vol}_g(\theta)$ be the integrated density of states. 
Then $N=\int_{\mathcal{M}} \frac{\mu_{\theta}+\mu_{f \theta}}{2}\ \mathrm{d}\mathrm{Vol}_g(\theta)$, so $N(U)>0$ for some set $U$ implies $\frac{\mu_{\theta}+\mu_{f\theta}}{2}(U)>0$ for $\mathrm{Vol}_g$-$a.e.\ \theta\in \mathcal{M}$.

\subsection{Rational approximation}
\subsubsection{Single frequency}
Let $\alpha$ be an irrational number and let $\{\frac{p_n}{q_n}\}$ be its continued fraction approximants. We have the following properties (see e.g.\cite{Continuedfrac}):
\begin{equation}\label{qnqn+1}
\frac{1}{2q_{n+1}}\leq \|q_n\alpha\|_{\T}\leq \frac{1}{q_{n+1}}.
\end{equation}
\begin{equation}\label{kqn}
 \|k\alpha\| >\|q_n \alpha\|\ \mathrm{for}\ q_n < k < q_{n+1}.
\end{equation}
\begin{enumerate}
\item If $\alpha\in DC(c, \tau)$ for some $c>0$, we have
\begin{equation}\label{Dioctau}
\|k\alpha\|_{\T}\geq \frac{c}{|k|^\tau}\ \ \mathrm{for}\ \mathrm{any}\ k\neq 0.
\end{equation}
In particular, combining (\ref{qnqn+1}) with (\ref{Dioctau}) we have
\begin{equation}\label{Dioqnqn+1}
c q_{n+1}\leq q_n^\tau.
\end{equation}
\item If $\alpha \notin DC(\tau)$, there exists a subsequence of the continued fraction approximants $\{\frac{p_{n_k}}{q_{n_k}}\}$ so that
\begin{equation}\label{NonDio}
q_{n_k+1}>q_{n_k}^\tau.
\end{equation}
\end{enumerate}
\subsubsection{Multiple frequencies}
Let $\alpha=(\alpha_1, \alpha_2,..., \alpha_d)$ be a set of irrational frequencies. Let $\{\frac{\vec{p}_n}{q_n}\}$ be its best simultaneous approximation with respect to the Euclidean norm on $\T^d$, namely, 
$$\sum_{j=1}^d{\|q_n\alpha_j\|_{\T}^2}<\sum_{j=1}^d{\|k \alpha_j\|_{\T}^2}\ \ \mathrm{for}\ \ \mathrm{any}\ |k|<q_n.$$
Clearly by the pigeonhole principle, we have
\begin{equation}\label{bestsimqnqn+1}
\sqrt{\sum_{j=1}^d{\|q_n\alpha_j\|_{\T}}^2}\leq \frac{2{\Gamma (\frac{d}{2}+1)}^{\frac{1}{d}}}{\sqrt{\pi}q_{n+1}^{\frac{1}{d}}}.
\end{equation}
We say that 
\begin{enumerate}
\item  $\alpha\in DC(c, \tau)$, if
\begin{equation}\label{multiDiotau}
\|\langle \vec{k}, \alpha \rangle\|_{\T}\geq \frac{c}{r(\vec{k})^\tau}\ \ \mathrm{for}\ \mathrm{any}\ \vec{k}\in \Z^d \backslash \{\vec{0}\}.
\end{equation}

\item  $\alpha\in WDC(c, \tau)$, if
\begin{equation}\label{multiweakDio}
\max_{1\leq j\leq d}{\|k \alpha_j\|_{\T}}\geq \frac{c}{|k|^\tau}\ \ \mathrm{for}\ \mathrm{any}\ k\in \Z\backslash \{\vec{0}\}.
\end{equation}
\end{enumerate}

\subsection{Discrepancy}
Let $\vec{x}_1,..., \vec{x}_N\in \mathcal{M}$. For a subset $C$ of $\mathcal{M}$, let $A(C;\{\vec{x}_n\})$ be the counting function 
\begin{align}\label{counting}
A(C;\{\vec{x}_n\}_{n=1}^N)=\sum_{n=1}^{N}\chi_{C}(\vec{x}_n)
\end{align}
The {\it isotropic discrepancy} $J_N(\{\vec{x}_n\}_{n=1}^N)$ is defined as
\begin{align}\label{defisodiscre}
J_N(\{\vec{x}_n\}_{n=1}^N)=\sup_{C\in \mathscr{C}}|\frac{A(C; \{\vec{x}_n\}_{n=1}^N}{N}-\mathrm{Vol}_g (C)|,
\end{align} 
where $\mathscr{C}$ is the family of all geodesically convex subsets of $\mathcal{M}$.

For a point $\theta\in \mathcal{M}$, let $J_N(\theta)=J(\{f^n\theta\}_{n=0}^{N-1})$. 
We say a map $f: \mathcal{M}\rightarrow \mathcal{M}$ has {\it{strongly $\delta$-bounded isotropic discrepancy}} if for some $N>N_0$, $J_N(\theta)\leq N^{-\delta}$ uniformly in $\theta\in \mathcal{M}$. 
We say $f$ has {\it{weakly $\delta$-bounded isotropic discrepany}} if there is a subsequence $\{N_j\}$ such that $J_{N_j}(\theta)\leq N_j^{-\delta}$ uniformly in $\theta\in \mathcal{M}$. 

If $\mathcal{M}=\T^d$ is the d-dimensional torus, we define the {\it discrepancy} $D_N(\{\vec{x}_n\}_{n=1}^N)$ as follows
\begin{align}\label{discrepancy}
D(\{\vec{x}_n\}_{n=1}^N)=\sup_{C\in \mathscr{J}} |\frac{A(C;\{\vec{x}\}_{n=1}^N)}{N}-m(C)|,
\end{align}
where $\mathscr{J}$ is the family of subintervals $C$ of the form $C=\{(\theta_1,...,\theta_d)\in \T^d: \beta_i\leq \theta_i<\kappa_i\ \mathrm{for} \ 1\leq i\leq d \}$.

For a point $\theta\in \T^d$, let $D_N(\theta)=D(\{f^n\theta\}_{n=0}^{N-1})$. 
We say a map $f: \T^d\rightarrow \T^d$ has {\it{strongly $\delta$-bounded discrepancy}} if for some $N>N_0$, $D_N(\theta)\leq N^{-\delta}$ uniformly in $\theta\in \T^d$. 
We will say $f$ has {\it{weakly $\delta$-bounded discrepany}} if there is a subsequence $\{N_j\}$ such that $D_{N_j}(\theta)\leq N_j^{-\delta}$ uniformly in $\theta\in \T^d$. 

When $\mathcal{M}=\T^d$, the isotropic discrepancy and discrepancy can be tightly controled by each other:
\begin{lemma}$\mathrm{(}$\cite{KN}, Theorem 1.6 in Chapter 2$\mathrm{)}$
For any sequence $\{\vec{x}_n\}_{n=1}^N$ in $\T^d$, we have
\begin{align}\label{disandisodis}
D_N(\{\vec{x}_n\}_{n=1}^N)\leq J_N(\{\vec{x}_n\}_{n=1}^N)\leq (4d\sqrt{d}+1)D_N(\{\vec{x}_n\}_{n=1}^N)^{\frac{1}{d}}.
\end{align}
\end{lemma}
Therefore, by (\ref{disandisodis}), when $\mathcal{M}=\T^d$, 
\begin{prop}\label{isoequ}
$f$ has strongly (weakly) $\delta$-bounded isotropic discrepancy for some $\delta>0$ if and only if $f$ has strongly (weakly) $\tilde{\delta}$-bounded discrepancy for some $\tilde{\delta}>0$.
\end{prop}

In section 5 and 6 we are going to apply the following two inequalities to estimate the discrepancy from above.
\begin{lemma}\cite{K-ETK} $\mathrm{[}$Erd\"os-Tur$\acute{a}$n-Koksma $\mathrm{inequality}\mathrm{]}$\label{ETK}
For any positive integer $H_0$, we have
\begin{equation}
D(\{\vec{x}_n\}_{n=1}^N)\leq C_d(\frac{1}{H_0}+\sum_{0<|\vec{h}|\leq H_0} \frac{1}{r(\vec{h})}|\frac{1}{N}\sum_{n=1}^N e^{2\pi i \langle \vec{h}, \vec{x}_n\rangle}|)
\end{equation}
where $|\vec{h}|=\max_{j=1}^d |h_j|$.
\end{lemma}

\begin{lemma}$\mathrm{(}$e.g. \cite{KN}, Lemma 3.1 in Chapter 1$\mathrm{)}$ $\mathrm{[}$Van der Corput's $\mathrm{Fundamental}\ \mathrm{Inequality}\mathrm{]}$\label{VDC}
For any integer $1\leq H\leq N$, we have
\begin{equation}
|\frac{1}{N}\sum_{n=1}^N u_n|^2\leq \frac{N+H-1}{N^2H}\sum_{n=1}^N |u_n|^2+\frac{2(N+H-1)}{N^2H^2}\sum_{k=1}^{H-1}(H-k) \mathrm{Re}\sum_{n=1}^{N-k}u_n\overline{u_{n+k}}.
\end{equation}
\end{lemma}

\section{key lemmas and proofs of Theorem \ref{xi} - \ref{beta2d}}
\subsection{Covering $\mathcal{M}$ with the orbit of a geodesic ball and proofs of Theorem \ref{xi}, \ref{xi2d}, \ref{beta} and \ref{beta2d}}
\begin{lemma}\label{stepxi}
Let $\phi$ be a piecewise H\"older function with $1\geq \gamma>0$. Suppose $L(E)$ is positive on a Borel subset $U$ with $N(U)>0$. 
\begin{enumerate}
\item If there exists a sequence $r_k\rightarrow 0$ so that any geodesic ball in $\mathcal{M}$ with radius $r_k$ covers the whole $\mathcal{M}$ in $r_k^{-M}$ steps, then $\underline{\xi_{\theta}}=0$ for $\mathrm{Vol}_g$-a.e. $\theta\in \mathcal{M}$;
\item If for any small $r>0$, any geodesic ball with radius $r$ covers the whole $\mathcal{M}$ in $r^{-M}$ steps, then $\overline{\xi_{\theta}}=0$ for $\mathrm{Vol}_g$-a.e. $\theta\in \mathcal{M}$;
\end{enumerate}
\end{lemma}
\begin{lemma}\label{stepbeta}
Let $\phi$ be a piecewise H\"older function with $1\geq \gamma>0$. Suppose $L(E)$ is continuous in $E$ and $L(E)>0$ for every $E\in \R$.
\begin{enumerate}
\item If there exists a sequence $r_k\rightarrow 0$ so that any geodesic ball in $\mathcal{M}$ with radius $r_k$ covers the whole $\mathcal{M}$ in $r_k^{-M}$ steps, then $\beta^-_{\theta}(p)=0$ for all $\theta\in \mathcal{M}$ and $p>0$;
\item If for any small $r>0$, any geodesic ball with radius $r$ covers the whole $\mathcal{M}$ in $r^{-M}$ steps, then $\beta^+_{\theta}(p)=0$ for all $\theta\in \mathcal{M}$ and $p>0$.
\end{enumerate}
\end{lemma}
Lemmas \ref{stepxi} and \ref{stepbeta} are key to our abstract argument. They are proved in section 4. The connection to bounded discrepancy comes in the following 

Let $r_g$ be as in Proposition \ref{volumegeoball} and $r_{g}^\prime$ as in Proposition \ref{ballconvex}.
\begin{lemma}\label{weakcube}
If $f$ has weakly $\delta$-bounded isotropic discrepancy, then there exists $r_k\rightarrow 0$ as $k\rightarrow \infty$ such that any geodesic ball in $\mathcal{M}$ with radius $r_k$ will cover the whole $\mathcal{M}$ in $r_k^{-\frac{2d}{\delta}}$ steps.
\end{lemma}

\begin{proof}
There exists a sequence $\{N_k\}$ and $k_0>0$ such that for any $k>k_0$ we have $J_{N_k}(\{f^n\theta\}_{n=0}^{N-1})\leq N_k^{-\delta}$. 
This means for any geodesically convex set $C\subset \mathcal{M}$, 
$\frac{\sum_{n=0}^{N_k-1}\chi_{C}(f^n\theta)}{N_k}-\mathrm{Vol}_g(C)\geq -N_k^{-\delta}$ holds for all $\theta\in \mathcal{M}$. 
Thus if we take $r_k=N_k^{-\frac{\delta}{2d}}<\min{(r_g, r_{g}^\prime)}$, then by Proposition \ref{ballconvex}, we know $B_{r_k}(\theta)$ is geodesically convex.
By Proposition \ref{volumegeoball}, $\mathrm{Vol}_g(B_{r_k}(\theta))\geq c_g r_k^{d}=c_gN_k^{-\frac{\delta}{2}}>N_k^{-\delta}$. 
Thus
$\sum_{n=0}^{r_k^{-\frac{2d}{\delta}}-1}\chi_{B_{r_k}(\theta)}(f^n\theta)>0$ for any $\theta\in \mathcal{M}$.
$\hfill{} \Box$
\end{proof}

\begin{lemma}\label{strongcube}
If $f$ has strongly $\delta$-bounded isotropic discrepancy, then for any $0<r<\min{(r_g, r_{g}^\prime)}$, any geodesic ball in $\mathcal{M}$ with radius $r$ will cover the whole $\mathcal{M}$ in $r^{-\frac{2d}{\delta}}$ steps.
\end{lemma}

\begin{proof}
There exists $N_0$ such that for any $N>N_0$ we have $J_N(\{f^n \theta\}_{n=0}^{N-1})\leq N^{-\delta}$ for all $\theta\in \mathcal{M}$. 
This means for any $0<r<\min{(r_g, r_{g}^\prime)}$, any geodesic ball $B_r(\theta)$ (it is geodesically convex by Proposition \ref{ballconvex}) and $N=r^{-\frac{2d}{\delta}}$ we have $\frac{\sum_{n=0}^{r^{-\frac{2d}{\delta}}-1}\chi_{B_r(\theta)}(f^n\theta)}{r^{-\frac{2d}{\delta}}}-\mathrm{Vol}_g(B_r(\theta))\geq -r^{2d}$. 
Since by Proposition \ref{volumegeoball}, $\mathrm{Vol}_g(B_r(\theta))\geq c_g r^{d}>r^{2d}$, we have $\sum_{n=0}^{r^{-\frac{2d}{\delta}}-1}\chi_{B_r(\theta)}(f^n\theta)>0$ for any $\theta\in \mathcal{M}$.     
$\hfill{} \Box$
\end{proof} 

In the case of 2-dimensional irrational rotation, we also have
\begin{lemma}\label{2dsteps}
For any $(\alpha_1,\alpha_2)\in \cup_{\tau>1}WDC(\tau)$, there exists $r_k(\alpha_1,\alpha_2,\tau)  \rightarrow 0$ as $k \rightarrow \infty$ such that any Euclidean ball with radius $r_k$ covers the whole $\T^2$ in $r_k^{-800\tau^4}$ steps.
\end{lemma}
\begin{rem}
This lemma will be proved in section 8.
\end{rem}
We are now ready to complete the proof of the main Theorems.
\subsection*{Proof of Theorem \ref{xi}}
Combining Lemma \ref{weakcube}, \ref{strongcube} with Lemma \ref{stepxi}. $\hfill{} \Box$
\subsection*{Proof of Theorem \ref{xi2d}}
Combining Lemma \ref{2dsteps} with Lemma \ref{stepxi}. $\hfill{} \Box$
\subsection*{Proof of Theorem \ref{beta}}
Combining Lemma \ref{weakcube}, \ref{strongcube} with Lemma \ref{stepbeta}. $\hfill{} \Box$
\subsection*{Proof of Theorem \ref{beta2d}}
Combining Lemma \ref{2dsteps} with Lemma \ref{stepbeta}. $\hfill{} \Box$

\subsection{Estimation of Discrepancy and proofs of Theorems \ref{xitoral}, \ref{betatoral}, \ref{xiskew} and \ref{betaskew}}
 We have the following control of the discrepancies of irrational rotation and skew-shift.
\begin{lemma}\label{toralsbdd}
If $\alpha\in DC(\tau)$, then for some constant $\delta>0$, $D_N(\{\theta+n\alpha\}_{n=0}^{N-1})\leq N^{-\delta}$ uniformly in ${\theta}\in \T^d$.
\end{lemma}
Let $\vec{Y}_n=(y_1+\binom{n}{1}\alpha,\ y_2+\binom{n}{1}y_1+\binom{n}{2} \alpha,\ ...,\ y_d+\binom{n}{1}y_{d-1}+\cdots +\binom{n}{d} \alpha)=f^n (y_1, \cdots, y_d)$, where $f$ is the skew shift.
\begin{lemma}\label{skewsbdd}
If $\alpha\in DC(\tau)$, then for some constant $\delta>0$, $D_N(\{\vec{Y}_n\}_{n=1}^N) \leq N^{-\delta}$ uniformly in $(y_1,...,y_d)\in \T^d$. 
\end{lemma}
\begin{lemma}\label{skewwbdd}
If $\alpha\notin DC(d)$, then for some constant $\delta>0$ there exists a sequence $\{N_j\}$ so that $D_{N_j}(\{\vec{Y}_n\}_{n=1}^{N_j})\leq N_j^{-\delta}$ uniformly in $(y_1,...,y_d)\in \T^d$.
\end{lemma}
\begin{rem}
Lemma \ref{toralsbdd} is standard. It's proof will be given in the appendix. The proofs of Lemma \ref{skewsbdd} and \ref{skewwbdd} will be given in section 6.
\end{rem}
\subsection*{Proof of Theorem \ref{xitoral}, \ref{betatoral}}
Follows from Lemma \ref{toralsbdd} and Theorems \ref{xi}, \ref{beta}. $\hfill{} \Box$
\subsection*{Proof of Theorem \ref{xiskew}, \ref{betaskew}}
Follows from Lemmas \ref{skewsbdd}, \ref{skewwbdd} and Theorems \ref{xi}, \ref{beta}. $\hfill{} \Box$

\section{Proofs of Lemmas \ref{stepxi} and \ref{stepbeta}}

\subsection{Upper and lower bounds on transfer matrices}
The following lemma on the uniform upper bound of transfer matrix is essentially from \cite{JM1}. We have adapted it into the following form for convenience.
\begin{lemma}\label{uniupp}$\mathrm{(}$\cite{JM1}, Theorem 3.1$\mathrm{)}$
Let $\phi$ be a function whose discontinuity set has measure $0$ and $f$ be a uniquely ergodic map on $\mathcal{M}$. Then
\begin{sublemma}\label{uniuppxi}
Let $L(E)$ be positive on a Borel set $U$ and $\mu$ be a measure with $\mu(U)>0$. Then for any $\zeta>0$ there exists a number $D_{\zeta}>0$, and for any $\epsilon>0$ there exists a set $B_{\zeta, \epsilon}$ with $0<\mu(B_{\zeta,\epsilon})<\zeta$,  and an integer $N_{\zeta, \epsilon}$ so that for any $E\in U\setminus B_{\zeta,\epsilon}$:
\begin{enumerate}
\item $L(E)\geq D_{\zeta}$,
\item for $n>N_{\zeta,\epsilon}$, $|z-E|<e^{-4\epsilon n}$ and $\theta\in \mathcal{M}$, we have $\frac{1}{n}\ln{\|A_{n}(\theta, z)\|}<L(E)+\epsilon$.
\end{enumerate}
\end{sublemma}
\begin{sublemma}\label{uniuppbeta}
Furthermore, if $L(E)$ is continuous in $E$ and $U$ is a compact set, there exists $D>0$ and for any $\epsilon>0$ there exists an integer $N_{\epsilon}$ so that for any $E\in U$:
\begin{enumerate}
\item $L(E)\geq D$
\item for $n>N_{\epsilon}$, $|z-E|<e^{-4\epsilon n}$ and $\theta \in \mathcal{M}$, we have $\frac{1}{n}\ln{\|A_n(\theta, z)\|}<L(E)+\epsilon$.
\end{enumerate}
\end{sublemma}
\end{lemma}

We are also able to formulate the following lower bound for the norm of transfer matrices.
\begin{lemma}\label{lowerbdd}
Let $\phi$ be a piecewise H\"older function with $1\geq \gamma>0$ and $f$ be a uniquely ergodic volume preserving map on $\mathcal{M}$ with $V(f)=V(f^{-1})=0$. 
Then
\begin{sublemma}\label{lowerbddxi}
Let $L(E)$ be positive on a Borel set $U$ and $\mu$ be a measure with $\mu(U)>0$. Then for any $\zeta, \epsilon>0$, let $D_{\zeta}$, $B_{\zeta, \epsilon}$ and $N_{\zeta, \epsilon}$ be defined as in $\ref{uniuppxi}$.
\begin{enumerate}
\item If there exists a sequence $r_k\rightarrow 0$ so that any geodesic ball in $\mathcal{M}$ with radius $r_k$ covers the whole $\mathcal{M}$ in $r_k^{-M}$ steps, then there exists a sequence $\{n_k(\epsilon)\}$ such that for $k>k_{\zeta,\epsilon}$, any $E\in U\setminus B_{\zeta, \epsilon}$, $|z-E|<e^{-4\epsilon n_k}$ and  $\theta\in \mathcal{M}$ we have
\begin{align*}
\min_{\iota\in\{-1,1\}} \max_{\iota j=0,..., e^{\frac{5M \epsilon}{\gamma}n_k}}\|A_{n_k}(f^j \theta, z)\|\geq e^{n_k(L(E)-3\epsilon)}.
\end{align*} 
\item If for any small $r>0$, any geodesic ball with radius $r$ covers the whole $\mathcal{M}$ in $r^{-M}$ steps, then for $n>N^\prime_{\zeta, \epsilon}$, any $E\in U\setminus B_{\zeta, \epsilon}$, $|z-E|<e^{-4\epsilon n}$ and  $\theta\in \mathcal{M}$ we have
\begin{align*}
\min_{\iota\in\{-1,1\}}  \max_{\iota j=0,..., e^{\frac{5M \epsilon}{\gamma}n}}\|A_n(f^j \theta, z)\|\geq e^{n(L(E)-3\epsilon)}.
\end{align*}
\end{enumerate}
\end{sublemma}
\begin{sublemma}\label{lowerbddbeta}
Furthermore, if $L(E)$ is continuous in $E$ and $U$ is a compact set, let $D$ be defined as in $\ref{uniuppbeta}$ and for any $\epsilon>0$ let $N_{\epsilon}$ be defined as in $\ref{uniuppbeta}$. Then for any $E\in U$ we have $L(E)\geq D$ and for any $|z-E|<e^{-4\epsilon n}$ we have
\begin{enumerate}
\item if there exists a sequence $r_k\rightarrow 0$ so that any geodesic ball in $\mathcal{M}$ with radius $r_k$ covers the whole $\mathcal{M}$ in $r_k^{-M}$ steps, then there exists a sequence $\{n_k(\epsilon)\}$ such that for $k>k_\epsilon$ and any $\theta\in \mathcal{M}$,
\begin{align*}
\min_{\iota\in\{-1,1\}}  \max_{\iota j=0,..., e^{\frac{5M \epsilon}{\gamma}n_k}}\|A_{n_k}(f^j \theta, z)\|\geq e^{n_k(L(E)-3\epsilon)}.
\end{align*} 
\item if for any small $r>0$, any geodesic ball with radius $r$ covers the whole $\mathcal{M}$ in $r^{-M}$ steps, then for $n>N^\prime_\epsilon$ and any $\theta\in \mathcal{M}$,
\begin{align*}
\min_{\iota\in\{-1,1\}}  \max_{\iota j=0,..., e^{\frac{5M \epsilon}{\gamma}n}}\|A_n(f^j \theta, z)\|\geq e^{n(L(E)-3\epsilon)}.
\end{align*}
\end{enumerate}
\end{sublemma}
\end{lemma}

\subsection*{Proof of Lemma \ref{lowerbdd}}
We will focus on the proof of part (1) of $\ref{lowerbddxi}$. The other three proofs will be discussed briefly at the end of this section.

For any $E\in U\setminus B_{\zeta, \epsilon}$ and $n>N_{\zeta,\epsilon}$, by Lemma $\ref{uniuppxi}$ we have $\frac{1}{n}\|A_n(\theta, E)\|<L(E)+\epsilon$. 
Since $\int_{\mathcal{M}} \frac{1}{n}\ln{\|A_n(\theta, E)\|}\ \mathrm{d}\mathrm{Vol}_g(\theta) \geq L(E)$, we have
\begin{equation} 
\mathrm{Vol}_g(M_{n,E,L(E),\epsilon}):=\mathrm{Vol}_g(\{\theta\in \mathcal{M} : \frac{1}{n}\ln{\|A_n(\theta, E)\|}> L(E)-\epsilon\})>\frac{1}{2}.
\end{equation}
Now we take any $\theta\in M_{n,E,L(E),\epsilon}$ and $|z-E|<e^{-4 \epsilon n}$. When $n>2N_{\zeta, \epsilon}+3$, by the standard telescoping we have,
\begin{align*}
\|A_n(\theta, z)\| 
& \geq \|A_n(\theta, E)\|-\|A_n(\theta,z)-A_n(\theta,E)\|\\
& \geq e^{n(L(E)-\epsilon)} - (n+2(N_{\zeta, \epsilon}+1)\|A\|_{\infty}^{N_{\zeta, \epsilon}})e^{n(L(E)-3\epsilon)}\\
& >e^{n(L(E)-2\epsilon)}
\end{align*}
for large enough $n>N^\prime_{\zeta, \epsilon}$. This means
\begin{equation}
M_{n,E,L(E),\epsilon}\subset M_{n,z,L(E),2\epsilon}.
\end{equation}
We know the discontinuity set of $\frac{1}{n}\ln{\|A_n(\theta,z)\|}$ is $J_n=\cup_{l=0}^{n-1}f^{-l}(J_{\phi})$, where $J_{\phi}=\cup_{j=1}^K \partial S_j$ is defined in section \ref{piecewiseholderf}.
By our assumption (\ref{d-1measurepartial}) and the fact the $V_{d-1}(f^{-1})=0$ (by the definition (\ref{scale}) of $V(f^{-1})$).
For $n$ large enough, we have 
\begin{align}\label{mJn}
\mathrm{Vol}_{g, d-1}(J_n)\leq e^{n\epsilon} \mathrm{Vol}_{g, d-1}(J_{\phi}),
\end{align}
note that the largeness depends only on $f$.
Define 
\begin{align*}
\tilde{M}_{n,z,L(E),2\epsilon}=M_{n,z,L(E),2\epsilon}\setminus \overline{F_{2e^{-5\epsilon n/{\gamma}}}(J_n)},
\end{align*} 
where a neighborhood is defined as
\begin{align*}
F_r(A)=\{\theta \in \mathcal{M}: dist(\theta, A)<r\}.
\end{align*}
Then by (\ref{mJn}),
\begin{align*}
\mathrm{Vol}_{g}(\tilde{M}_{n,z,L(E),2\epsilon})&\geq \mathrm{Vol}_g(M_{n,z,L(E),2\epsilon})-4 e^{-5\epsilon n/\gamma}\mathrm{Vol}_{g,d-1}(J_n)\\
&\geq \mathrm{Vol}_g(M_{n,z,L(E),2\epsilon})-4 e^{-n(\frac{5\epsilon}{\gamma}-\epsilon)} \mathrm{Vol}_{g, d-1}(J_{\phi})>\frac{2}{5}.
\end{align*}
In particular, it is a non-empty set. Now we take any $\tilde{\theta} \in \tilde{M}_{n,z,L(E),2\epsilon}$ and $\theta\in B_{e^{-5\epsilon n/\gamma}}(\tilde{\theta})$. 
We have, by telescoping, (\ref{differencenorm}) and the fact that $V_1(f)=0$ (by the definition (\ref{scale}) of $V(f)$),
\begin{align*}
\|A_n(\theta, z)\|
&\geq \|A_n(\tilde{\theta},z)\|-\|A_n(\theta,z)-A_n(\tilde{\theta},z)\|\\
&\geq e^{n(L(E)-2\epsilon)}-(\sum_{l=1}^K\|\phi_l\|_{L^\gamma})(n+2(N_{\zeta, \epsilon}+1)\|A\|_{\infty}^{N_{\zeta, \epsilon}})e^{n(L(E)+\epsilon)}\max_{j=0,..., n-1}(dist(f^j\theta, f^j\tilde{\theta}))^{\gamma} \\
&\geq e^{n(L(E)-2\epsilon)}-(\sum_{l=1}^K\|\phi_l\|_{L^\gamma})(dist(\theta, \tilde{\theta}))^{\gamma} (n+2(N_{\zeta, \epsilon}+1)\|A\|_{\infty}^{N_{\zeta, \epsilon}})e^{n(L(E)+\epsilon+\gamma\epsilon)}\\
&>e^{n(L(E)-3\epsilon)}.
\end{align*}
for $n>N^{\prime \prime}_{\zeta, \epsilon}$. 
This means 
\begin{align*}
F_{e^{-5\epsilon n/{\gamma}}}(\tilde{M}_{n,z,L(E), 2\epsilon})\subset M_{n,z,L(E), 3\epsilon}.
\end{align*}
Hence for $E\in U\setminus B_{\zeta,\epsilon}$, $n>N^{\prime \prime}_{\zeta, \epsilon}$ and $|z-E|<e^{-4\epsilon n}$, $M_{n,z,L(E),3\epsilon}$ contains a geodesic ball with radius $e^{-\frac{5\epsilon}{\gamma}n}$.
Then there exists a sequence $\{n_k(\epsilon)\}$ such that a geodesic ball with radius $e^{-\frac{5\epsilon}{\gamma}n_k}\sim r_k$ covers the whole $\mathcal{M}$ in at most $e^{\frac{5M \epsilon}{\gamma}n_k}$ steps. 
Thus for $E\in U\setminus B_{\zeta, \epsilon}$, $k>k_{\zeta, \epsilon}$ so that $n_k(\epsilon)>N^{\prime \prime}_{\zeta, \epsilon}$, any $|z-E|<e^{-4\epsilon n_k}$ and any $\theta\in \T^d$ we have
\begin{align*}
\min_{\iota\in\{-1,1\}}  \max_{\iota j=0,..., e^{\frac{5M \epsilon}{\gamma}n_k}} \|A_{n_k} (f^j \theta, z)\|>e^{n_k(L(E)-3\epsilon)}.
\end{align*}

\begin{rem}
Notice that part (2) of Lemma $\ref{lowerbddxi}$ follows without taking a subsequence $\{n_k(\epsilon)\}$. Also, $\ref{lowerbddbeta}$ follows without excluding the set $B_{\zeta,\epsilon}$.
\end{rem}
$\hfill{} \Box$

\subsection{Dynamical bounds on $\xi_{\theta}$}
The key to estimate $\xi_{\theta}$ is to apply the following lemma by Killip, Kiselev and Last.

Following \cite{JL1}, for $f: \Z\rightarrow H$ where $H$ is a Banach space, the truncated $l^2$ norms in the positive and negative directions are defined by
\begin{align*}
\|f\|_{L}^2=\sum_{n=1}^{\lfloor L \rfloor} |f(n)|^2+ (L-\lfloor L \rfloor) |f(\lfloor L \rfloor +1)|^2\ \mathrm{for}\ L>0\\
\|f\|_{L}^2=\sum_{n=0}^{\lfloor L \rfloor +1} |f(n)|^2+ (\lfloor L \rfloor+1-L)|f(\lfloor L \rfloor)|^2\ \mathrm{for}\ L<0
\end{align*}
The truncated $l^2$ norm in both directions is defined by
\begin{align*}
\|f\|_{L_1, L_2}^2=\sum_{n=-\lfloor L_1 \rfloor}^{\lfloor L_2 \rfloor} |f(n)|^2+ (L_1-\lfloor L_1 \rfloor) |f(-\lfloor L_1 \rfloor-1)|^2+(L_2-\lfloor L_2 \rfloor) |f(\lfloor L_2 \rfloor+1)|^2\ \mathrm{for}\ L_1,L_2\geq 1.
\end{align*}

With $A_{\bullet}(\theta,z)$ being a function on $\Z$, define $\tilde{L}^{+}_\epsilon(\theta,z)\in \R^{+}$ and $\tilde{L}^{-}_\epsilon(\theta, z)\in \R^{-}$ by requiring
\begin{align*}
\|A_{\bullet}(\theta,z)\|_{\tilde{L}^{\pm}_\epsilon(\theta, z)}=2\|A(\theta,z)\|\epsilon^{-1}.
\end{align*}

\begin{lemma}\label{Last}$\mathrm{(}$\cite{KKL}, Theorem 1.5$\mathrm{)}$
Let $H_{\theta}$ be a Schr\"odinger operator and $\mu_{\theta}$ be the spectral measure of $H_{\theta}$ and $\delta_0$. Let $T>0$ and $L_1, L_2>2$, then
\begin{align}\label{Lastequ}
\langle \frac{1}{2}( \|e^{-it H_{\theta}}\delta_0\|^2_{L_1,L_2}+ \|e^{-it H_{\theta}}\delta_1\|^2_{L_1,L_2})\rangle_T>C\frac{\mu_{\theta}+\mu_{f\theta}}{2} (\{E: |\tilde{L}^{-}_{T^{-1}}|\leq L_1; \tilde{L}^{+}_{T^{-1}}\leq L_2\})
\end{align}
where $C$ is an universal constant \footnotemark. \footnotetext{Here we formulate this Lemma for operators with potential $V(n)=\phi(f^n \theta)$. This covers arbitrary bounded potentials by taking $f$ to be a corresponding subshift.}
\end{lemma}

This lemma directly implies $P_{\theta, T}(L)+P_{f\theta, T}(L)>C\frac{\mu_{\theta}+\mu_{f\theta}}{2}(\{E: \|A_{\bullet}(\theta,z)\|_{\pm L}>2\|A(\theta,z)\| T\})$. 
The plan is to show that for any $\eta>1$, any $\theta_0$ satisfying $(\mu_{\theta_0}+\mu_{f\theta_0})(U)>0$, we have $(\mu_{\theta_0}+\mu_{f\theta_0})(\{E:\|A_{\bullet}(\theta_0, z)\|_{\pm T}>T^{\eta}\})\gtrsim (\mu_{\theta_0}+\mu_{f\theta_0})(U)$.

\subsection*{Proof of Lemma \ref{stepxi}}
We will prove part (1) in detail. Part (2) will be discussed briefly at the end of this proof.

Fix $\eta>1$. Fix $\theta_0$ such that $(\mu_{\theta_0}+\mu_{f\theta_0})(U)>0$. 
Let $\zeta=\frac{1}{2}(\mu_{\theta_0}+\mu_{f\theta_0})(U)$, so a constant.
Let $D=D_{\zeta}$ from Lemma \ref{uniupp}.
Let $\epsilon=\min{(\frac{\gamma D}{40M \eta}, \frac{D}{6})}$. Then by Lemmas \ref{uniupp}, there exists a set $B$, $0<|B|<\frac{1}{2}(\mu_{\theta_0}+\mu_{V\theta_0})(U)$, and a sequence $\{n_k\}$, s.t. $L(E)\geq D$ on $U\setminus B$ and for $E\in U\setminus B$, $k\geq k_0$, $|z-E|<e^{-4\epsilon n_k}$ and any $\theta\in \mathcal{M}$,
\begin{align*}
\min_{\iota\in\{-1,1\}}  \max_{\iota j=0,..., e^{\frac{5M \epsilon}{\gamma}n_k}}\|A_{n_k}(f^j \theta, z)\|>e^{n_k(L(E)-3\epsilon)}.
\end{align*}
Using that $A_{s+t}(\theta, z)=A_{t}(f^{s}(\theta), z)A_{s}(\theta, z)$, this implies, by the condition on $\epsilon$,
\begin{align*}
\|A_{\bullet}(\theta, z)\|_{\pm e^{\frac{10M \epsilon}{\gamma}n_k}}>e^{\frac{n_k(L(E)-3\epsilon)}{2}}\geq e^{\frac{10M \epsilon}{\gamma}n_k\eta}.
\end{align*}
If we take $T_k=e^{\frac{10 M\epsilon}{\gamma}n_k}$, then $U\setminus B \subset \{E: \|A_{\bullet}(\theta, E)\|_{\pm T_k}>T_k^{\eta}\}$ for any $\theta$, in particular $\theta_0$. 
Then by (\ref{Lastequ}),
\begin{align*}
P_{\theta_0, T_k^\eta}(T_k)+P_{f\theta_0, T_k^\eta}(T_k)\geq C \frac{\mu_{\theta_0}+\mu_{f \theta_0}}{2}(\{E: \|A_{\bullet}(\theta_0, E)\|_{\pm T_k}>T_k^\eta\})\geq \tilde{C}\frac{\mu_{\theta_0}+\mu_{f \theta_0}}{2}(U).
\end{align*}
This implies $\underline{\xi_{\theta}}=0$ for all $\theta\in \mathcal{M}$ such that $(\mu_{\theta}+\mu_{f\theta})(U)>0$.

\begin{rem}
Using Lemmas \ref{uniuppxi} (2), \ref{lowerbddxi} (2) instead of \ref{uniuppxi} (1), \ref{lowerbddxi} (1), Part (2) can be proved without taking a subsequence $n_k$ therefore the conclusion holds for all $T$ large enough rather than a sequence $T_k$.  $\hfill{}\ \Box$
\end{rem}

\subsection{Bounds on $\beta$}
The key to the bounds on $\beta$ is to apply the following lemma by Damanik and Tcheremchansev.
\begin{lemma}\label{Damanik}$\mathrm{(}$Theorem 1 of \cite{DT1} plus Corollary 1 of \cite{DT2}$\mathrm{)}$
Let $H$ be the Schr$\ddot{o}$dinger operator, with $f$ real valued and bounded, and $K\geq 4$ such that $\sigma(H)\subset [-K+1, K-1]$. Suppose for all $\rho \in (0,1)$ we have
\begin{align}\label{Damequ}
\int_{-K}^K \left(\min_{\iota \in\{-1,1\}} \max_{1\leq \iota n\leq T^\rho}\|A_n(E+\frac{i}{T})\|^2\right)^{-1}\mathrm{d}E=O(T^{-\eta}).
\end{align}
for any $\eta\geq 1$. Then $\beta^+(p)=0$ for all $p>0$. If (\ref{Damequ}) is satisfied for a sequence $T_k\rightarrow \infty$, then $\beta^-(p)=0$ for all $p>0$.
\end{lemma}

\subsection*{Proof of Lemma \ref{stepbeta}}
We will prove part (1) in detail. A modification needed for part (2) is discussed briefly at the end of this proof.

It suffices to consider small $\rho\in (0,1)$. Fix any $\rho \in (0,1)$ small and $\eta\geq 1$. 
Aussme $\sigma(H)\subset [-K+1,K-1]$. 
Since $L(E)$ is continous in $E$ on a compact set $ [-K,K]$, we have $L(E)\geq D>0$ on $[-K, K]$.
Fix $\epsilon_{\eta}=\min{(\frac{\rho \gamma D}{20 M \eta},\frac{D}{6})}$. 
By Lemma $\ref{lowerbddbeta}$ there exists a sequence $\{n_{\eta, k}\}$ such that for any $E\in [-K,K]$, $k>k_\eta$, any $|z-E|<e^{-4\epsilon_{\eta} n_{\eta, k}}$ and any $\theta\in \mathcal{M}$,
\begin{align*}
\min_{\iota\in\{-1,1\}}  \max_{\iota j=0,..., e^{\frac{5M \epsilon_{\eta}}{\gamma}n_{\eta,k}}}\|A_{n_{\eta,k}}(f^j \theta, z)\|>e^{n_{\eta,k}(L(E)-3\epsilon_{\eta})}.
\end{align*}
Thus
\begin{align*}
\min_{\iota\in \{-1,1\}} \max_{j=0,...,e^{\frac{10M \epsilon_{\eta}}{\gamma}n_{\eta,k}}}\|A_j(\theta, z)\|^2 \geq e^{n_{\eta, k}(L(E)-3\epsilon_{\eta})}\geq e^{\frac{10M \epsilon_{\eta}}{\gamma \rho} n_{\eta,k}\eta}
\end{align*}
holds for any $\theta \in \mathcal{M}$, any $E\in [-K,K]$ and $|z-E|<e^{-4\epsilon_{\eta} n_{\eta,k}}$. 
Now we take $T_{\eta ,k}=e^{\frac{10M \epsilon_{\eta}}{\gamma \rho}n_{\eta,k}}$,
\begin{align*}
|E+\frac{i}{T_{\eta,k}}-E|=\frac{1}{T_{\eta,k}}<e^{-4\epsilon_{\eta} n_{\eta,k}}.
\end{align*}
Thus
\begin{align*}
\min_{\iota\in \{-1, 1\}} \max_{\iota j=0,...,T_{\eta,k}^\rho}\|A_j(\theta, E+\frac{i}{T_{\eta,k}})\|^2 \geq T_{\eta,k}^{\eta}
\end{align*} 
holds for any $E\in[-K,K]$.
Therefore 
\begin{align*}
\int_{-K}^K \left(\min_{\iota \in\{-1,1\}} \max_{1\leq \iota n\leq T_{\eta,k}^\rho}\|A_n(\theta,E+\frac{i}{T_{\eta,k}})\|^2\right)^{-1}\ dE \leq 2K T_{\eta,k}^{-\eta}.
\end{align*}
Now take a sequence $\{k_i\}$ such that $T_{1,k_1}<T_{2,k_2}<...$ Let $T_m=T_{m,k_m}$. Then
\begin{align*}
\int_{-K}^K \left(\min_{\iota \in \{-1,1\}} \max_{1\leq \iota n\leq T_m^{\rho}}\|A_n(\theta,E+\frac{i}{T_m})\|^2\right)^{-1}\ dE\leq 2K T_m^{-m}.
\end{align*}
By (\ref{Damequ}), we have $\beta_{\theta}^-(p)\leq \rho$ for all $\theta \in \mathcal{M}$, any $\rho\in (0,1)$ and any $p>0$, thus $\beta_{\theta}^-(p)=0$ for all $\theta \in \mathcal{M}$ and any $p>0$.
\begin{rem}
Using Lemmas \ref{uniuppbeta} (2) and \ref{lowerbddbeta} (2), part (2) follows without taking a subsequence $\{n_{\eta, k}\}$. Therefore the conclusion holds for all $T$ large rather than a sequence $T_k$. $\hfill{}\ \Box$
\end{rem}

\section{Skew-shift. Proof of Lemmas \ref{skewsbdd} and \ref{skewwbdd}}
In this section, we obtain the discrepency bounds for the skew shift. While the Diophantine case is likely known, we didn't find this in the literature. We thus present a detailed proof, especially since we build our proof for the Liouvillian case on some of the same considerations.

\subsection*{Skew-shift}
Let $f$: $\T^d\rightarrow \T^d$ be defined as follows
\begin{equation*}
f(y_1,y_2,...,y_d)=(y_1+\alpha, y_2+y_1,...,y_d+y_{d-1}).
\end{equation*}
Let $\vec{Y}_n=f^n(y_1,...,y_d)$, then
\begin{equation}\label{Yn}
\vec{Y}_n=(y_1+\binom{n}{1}\alpha,\ y_2+\binom{n}{1}y_1+\binom{n}{2} \alpha,\ ...,\ y_d+\binom{n}{1}y_{d-1}+\cdots +\binom{n}{d} \alpha),
\end{equation}
where $\binom{n}{m}=0$ if $n<m$.

\subsection{Preparation. Combinatorial identities }

\begin{lemma}
Let $r_t\in \N$ for $1\leq t\leq s$, then we have
\begin{align}
\sum_{1\leq t\leq s}^{l_t=0,1} (-1)^{s-\sum_{t=1}^s l_t} \binom{\sum_{t=1}^s l_t r_t}{s-1}&=0, \label{s-1comb}\\
\sum_{1\leq t\leq s}^{l_t=0,1} (-1)^{s-\sum_{t=1}^s l_t} \binom{\sum_{t=1}^s l_t r_t}{s}&=\prod_{t=1}^s r_t.\label{scom}
\end{align}
\end{lemma}

\begin{proof}
Let us consider the coefficient $C_{a}$ of $x^{a}$ in the product 
$(1+x)^{r_1}\cdot (1+x)^{r_2}\cdot \cdots \cdot (1+x)^{r_s}=(1+x)^{\sum_{i=1}^{s}r_i}$.
Let us denote 
\begin{align}\label{s-1combA}
A^{(a)}=\{(\vec{j}_1, \vec{j}_2,..., \vec{j}_s), \mathrm{where}\ \vec{j}_{t}=(j_{t,1}, j_{t,2},..., j_{t, r_t}),\ j_{t, k}\in\{0, 1\}| \sum_{t=1}^s\sum_{k=1}^{r_t}j_{t,k}=a\}
\end{align}
Each element in $A^{(a)}$ corresponds to one way of choosing $1$ or $x$ in each term of the product $(1+x)^{r_1}\cdot (1+x)^{r_2}\cdot \cdots \cdot (1+x)^{r_s}$ in order to get $x^{a}$, where $j_{t, k}=0$ means we choose $1$ out of the $k$-th $1+x$ from $(1+x)^{r_t}$, and $j_{t, k}=1$ means we choose $x$ instead of $1$. 
Thus the capacity of $A^{(a)}$, denoted by $|A^{(a)}|$, is equal to $C_\alpha=\binom{\sum_{t=1}r_t}{a}$.
Let us futher denote
\begin{align}\label{s-1combAt}
A_t^{(a)}=A^{(a)}\cap \{\vec{j}_t=\vec{0}\}
\end{align}
For $a=s-1$, since it is impossible to obtain $x^{s-1}$ with $\vec{j}_t\neq \vec{0}$ for any $1\leq t\leq s$, we have
\begin{align}\label{s-1union}
A^{(s-1)}\setminus (\cup_{t=1}^s A_t^{(s-1)})=\emptyset.
\end{align} 
For $a=s$,
\begin{align}\label{sunion}
A^{(s)}\setminus(\cup_{t=1}^s A_t^{(s)})=D,
\end{align} 
where 
\begin{align}\label{scombD}
D=\{(\vec{j}_1, \vec{j}_2,..., \vec{j}_t)|\sum_{k=1}^{r_t}j_{t,k}=1\ \mathrm{for}\ 1\leq t\leq s\}.
\end{align}
Clearly,
\begin{align}\label{s-1combAAt}
|\cup_{t=1}^{s}A_t^{(a)}|=\sum_{i=1}^s (-1)^{i-1} \sum_{1\leq t_1<t_2<\cdots <t_i\leq s}|\cap_{l=1}^i A_{t_l}^{(a)}|,
\end{align}
in which
\begin{align}\label{sumlAtk}
\sum_{1\leq t_1<t_2<\cdots <t_i\leq s}|\cap_{l=1}^i A_{t_l}^{(a)}|=\sum_{\sum_{t=1}^s l_t=s-i}^{l_t=0,1}\binom{\sum_{t=1}^s l_t r_t}{a}.
\end{align}
Thus
\begin{align}\label{Aa}
|A^{(a)}\setminus (\cup_{t=1}^{s}A_t^{(a)})| &=\binom{\sum_{t=1}^sr_t}{a}+\sum_{i=1}^{s}(-1)^{i}\sum_{\sum_{t=1}^s l_t=s-i}^{l_t=0,1}\binom{\sum_{t=1}^s l_t r_t}{a}, \notag\\
&=\sum_{1\leq t \leq s}^{ l_t=0,1} (-1)^{s-\sum_{t=1}^{s}l_t} \binom{\sum_{t=1}^s l_t r_t}{a}.
\end{align}
For $a=s-1$, (\ref{s-1comb}) follows directly from  (\ref{s-1union}) and (\ref{Aa}).
For $a=s$, (\ref{scom}) follows from (\ref{sunion}), (\ref{Aa}) and the fact that $|D|=\prod_{t=1}^s r_t$.    $\hfill{} \Box$
\end{proof}

\subsection{Diophantine $\alpha$. Proof of Lemma \ref{skewsbdd}}
For $\alpha \in DC(\tau)$, we take integers
\begin{equation}\label{H_j}
H_j \sim N^{\frac{2^j}{(2^d-1)(\tau+\epsilon)}}\ \mathrm{for}\   0 \leq j \leq d-1.
\end{equation}
By Lemma $\ref{ETK}$,
\begin{align}\label{alphaETK}
D(\vec{Y}_1,..., \vec{Y}_N) 
& \leq C_d (\frac{1}{H_0}+\sum_{0< |\vec{h}|\leq H_0} \frac{1}{r(\vec h)} |\frac{1}{N}\sum_{n=1}^N e^{2\pi i \langle \vec{h}, \vec{Y}_n \rangle}|) \notag \\
&=C_d (\frac{1}{H_0}+\sum_{0< |\vec{h}|\leq H_0} \frac{1}{r(\vec h)} |\frac{1}{N}\sum_{n=1}^N u_n^{(0)}| ) ,
\end{align}
where 
\begin{align}\label{u0}
u_n^{(0)}=\exp\{2\pi i \sum_{j=1}^{d} (h_j\alpha+ \sum_{r=1}^{d-j}h_{j+r} y_r) \binom{n}{j}\}.
\end{align}

For $1\leq s\leq d-2$, let 
\begin{align}\label{us}
u^{(s)}_{k_1,..., k_s, n}=\exp\Big \{2\pi i \sum_{j=s+1}^d (h_j\alpha+ \sum_{r=1}^{d-j}h_{j+r} y_r)\sum_{1\leq t\leq s}^{l_t=0,1}(-1)^{s-\sum_{t=1}^s l_t}\binom{n+\sum_{t=1}^s l_t k_t}{j}\Big \}
\end{align}
Then by Lemma $\ref{VDC}$,
\begin{align}\label{uss+1}
& |\frac{1}{N-\sum_{t=1}^s k_s}\sum_{n=1}^{N-\sum_{t=1}^s k_t}u_{k_1,..., k_s, n}^{(s)}|^2\\
\lesssim & \frac{1}{H_{s+1}}+\frac{1}{(N-\sum_{t=1}^{s}k_t)H_{s+1}^2}\sum_{k_{s+1}=1}^{H_{s+1}}(H_{s+1}-k_{s+1})|\sum_{n=1}^{N-\sum_{t=1}^{s+1}k_t}u_{k_1,..., k_s, n}^{(s)}\overline{u^{(s)}_{k_1,..., k_s, n+k_{s+1}}}|. \notag
\end{align}
Here 
\begin{align}
&|\sum_{n=1}^{N-\sum_{t=1}^{s+1}k_t}u_{k_1,..., k_s, n}^{(s)}\overline{u^{(s)}_{k_1,..., k_s, n+k_{s+1}}}| \notag\\
=&|\sum_{n=1}^{N-\sum_{t=1}^{s+1}k_t}\exp \Big \{2\pi i \sum_{j=s+1}^d (h_j\alpha+ \sum_{r=1}^{d-j}h_{j+r} y_r)\sum_{1\leq t\leq s}^{l_t=0,1}(-1)^{s-\sum_{t=1}^sl_t}\left(\binom{n+\sum_{t=1}^{s}l_tk_t}{j}-\binom{n+k_{s+1}+\sum_{t=1}^s l_tk_t}{j}\right)\Big \}| \notag\\
=&|\sum_{n=1}^{N-\sum_{t=1}^{s+1}k_t}\exp \Big \{2\pi i \sum_{j=s+1}^d (h_j\alpha+ \sum_{r=1}^{d-j}h_{j+r} y_r)\sum_{1\leq t\leq s+1}^{l_t=0,1}(-1)^{s+1-\sum_{t=1}^{s+1} l_t}\binom{n+\sum_{t=1}^{s+1}l_tk_t}{j} \Big \}| \notag\\
=&|\sum_{n=1}^{N-\sum_{t=1}^{s+1}k_t}\exp \Big \{2\pi i \sum_{j=s+1}^d (h_j\alpha+ \sum_{r=1}^{d-j}h_{j+r} y_r)\sum_{0\leq t\leq s+1}^{l_t=0,1}(-1)^{s+2-\sum_{t=0}^{s+1} l_t}\binom{l_0 n+\sum_{t=1}^{s+1}l_tk_t}{j} \Big \}|\notag\\
=&|\sum_{n=1}^{N-\sum_{t=1}^{s+1}k_t}\exp \Big \{2\pi i \sum_{j=s+2}^d (h_j\alpha+ \sum_{r=1}^{d-j}h_{j+r} y_r)\sum_{0\leq t\leq s+1}^{l_t=0,1}(-1)^{s+2-\sum_{t=0}^{s+1} l_t}\binom{l_0 n+\sum_{t=1}^{s+1}l_tk_t}{j} \Big \}| \label{uss+1detail2}\\
=&|\sum_{n=1}^{N-\sum_{t=1}^{s+1}k_t}\exp \Big \{2\pi i \sum_{j=s+2}^d (h_j\alpha+ \sum_{r=1}^{d-j}h_{j+r} y_r)\sum_{1\leq t\leq s+1}^{l_t=0,1}(-1)^{s+1-\sum_{t=1}^{s+1} l_t}\binom{n+\sum_{t=1}^{s+1}l_tk_t}{j} \Big \}|\notag\\
=&|\sum_{n=1}^{N-\sum_{t=1}^{s+1}k_t}u^{(s+1)}_{k_1,..., k_{s+1}, n}|.\label{uss+1detail}
\end{align}
Notice that in (\ref{uss+1detail2}), we applied (\ref{scom}),
\begin{align*}
\exp \Big \{ (h_{s+1}\alpha+ \sum_{r=1}^{d-s-1}h_{s+1+r} y_r)\sum_{0\leq t\leq s+1}^{l_t=0,1}(-1)^{s+2-\sum_{t=0}^{s+1} l_t}\binom{l_0 n+\sum_{t=1}^{s+1}l_tk_t}{s+1} \Big \}=1.
\end{align*}
Combining (\ref{uss+1}) with (\ref{uss+1detail}), we get for any $0\leq s\leq d-3$,
\begin{align}\label{realuss+1}
       & |\frac{1}{N-\sum_{t=1}^s k_s}\sum_{n=1}^{N-\sum_{t=1}^s k_t}u_{k_1,..., k_s, n}^{(s)}|^2\\
\leq &\frac{1}{H_{s+1}}+\frac{1}{(N-\sum_{t=1}^{s}k_t)H_{s+1}^2}\sum_{k_{s+1}=1}^{H_{s+1}}(H_{s+1}-k_{s+1})(N-\sum_{t=1}^{s+1}k_t)|\frac{1}{N-\sum_{t=1}^{s+1}k_t}\sum_{n=1}^{N-\sum_{t=1}^{s+1}k_t}u^{(s+1)}_{k_1,..., k_{s+1}, n}|. \notag 
\end{align}
By (\ref{uss+1}), for $s=d-2$,
\begin{align}\label{alphad-2}
&|\frac{1}{N-\sum_{l=1}^{d-2}k_l} \sum_{n=1}^{N-\sum_{l=1}^{d-2}k_l} u_{k_1,..., k_{d-2},n}^{(d-2)} |^2 \\
\lesssim &\frac{1}{H_{d-1}} + \frac{1}{(N-\sum_{l=1}^{d-2}k_l)H_{d-1}^2} \sum_{k_{d-1}=1}^{H_{d-1}} (H_{d-1}-k_{d-1}) |\sum_{n=1}^{N-\sum_{l=1}^{d-1}k_l} u_{k_1,..., k_{d-2},n}^{(d-2)} \overline{u_{k_1,..., k_{d-2},n+k_{d-1}}^{(d-2)}} |    \notag\\
\lesssim &\frac{1}{H_{d-1}} + \frac{1}{(N-\sum_{l=1}^{d-2}k_l)H_{d-1}} \sum_{k_{d-1}=1}^{H_{d-1}} |\sum_{n=1}^{N-\sum_{l=1}^{d-1}k_l} u_{k_1,..., k_{d-2},n}^{(d-2)} \overline{u_{k_1,..., k_{d-2},n+k_{d-1}}^{(d-2)}} |, \notag 
\end{align}
and
\begin{align}
   &|\sum_{n=1}^{N-\sum_{l=1}^{d-1}k_l} u_{k_1,..., k_{d-2},n}^{(d-2)} \overline{u_{k_1,...,k_{d-2},n+k_{d-1}}^{(d-2)}} | \notag\\
= &|\sum_{n=1}^{N-\sum_{l=1}^{d-1}k_l} \exp\{2\pi i h_d\alpha \sum^{j_l=0,1}_{1\leq l \leq d-1}(-1)^{d-1-\sum_{l=1}^{d-1}j_l} 
\binom{n+\sum_{j=1}^{d-1} j_l k_l}{d} \} | \notag \\
= &|\sum_{n=1}^{N-\sum_{l=1}^{d-1}k_l} \exp\{2\pi i h_d\alpha \sum^{j_l=0,1}_{0\leq l \leq d-1}(-1)^{d-\sum_{l=0}^{d-1}j_l} 
\binom{l_0 n+\sum_{j=1}^{d-1} j_l k_l}{d} \} | \notag\\
= &|\sum_{n=1}^{N-\sum_{l=1}^{d-1}k_l} \exp\{2\pi i h_d n \alpha \prod_{l=1}^{d-1}k_l\}|  \label{d-1estimatedetail}   \\
\lesssim & \frac{1}{\|h_d \alpha \prod_{l=1}^{d-1} k_l\|_{\T}}, \label{estimated-1}
\end{align}
where in (\ref{d-1estimatedetail}) we used (\ref{scom}).

Since $\alpha\in DC(\tau)$, by the property of Diophantine condition (\ref{Dioctau}) and since $|h_i|\leq H_0$, $1\leq k_i\leq H_i$ we have
\begin{equation}\label{alphatau}
\sum_{k_{d-1}=1}^{H_{d-1}}\frac{1}{\|h_d \alpha \prod_{l=1}^{d-1} k_l \|_{\T}}\leq \sum_{j=1}^{H_{d-1}}\frac{m^{\tau}\prod_{l=1}^{d-1}H_l^{\tau}}{j}\leq m^{\tau}  H_{d-1}^{\tau+\epsilon} \prod_{l=1}^{d-2} H_l^{\tau}.
\end{equation}
Thus combining (\ref{alphad-2}), (\ref{estimated-1}) with (\ref{alphatau}), we have
\begin{equation*}
|\frac{1}{N-\sum_{l=1}^{d-2}k_l}\sum_{n=1}^{N-\sum_{l=1}^{d-2}k_l}u^{(d-2)}_{k_1,...,k_{d-2},n}|^2 \lesssim \frac{1}{H_{d-1}}+\frac{m^{\tau} H_{d-1}^{\tau+\epsilon} \prod_{l=1}^{d-2}H_l^{\tau} }{H_{d-1}(N-\sum_{l=1}^{d-2}H_l)}\lesssim \frac{1}{H_{d-1}}=\frac{1}{H_{d-2}^2}.
\end{equation*}
\begin{lemma}\label{reverse}
For any $\alpha\in \T$, if for any $1\leq k_s\leq H_s$,
\begin{align*}
|\frac{1}{N-\sum_{l=1}^{s}k_l}\sum_{n=1}^{N-\sum_{l=1}^{s}k_l}u^{(s)}_{k_1,...,k_{s},n}|^2\lesssim \frac{1}{H_s^2},
\end{align*} 
then for any $0\leq t\leq s-1$, $1\leq k_t\leq H_t$ we have 
\begin{align*}
|\frac{1}{N-\sum_{l=1}^{t}k_l}\sum_{n=1}^{N-\sum_{l=1}^{t}k_l}u^{(t)}_{k_1,...,k_{t},n}|^2\lesssim \frac{1}{H_t^2}.
\end{align*}
\end{lemma}
\begin{proof}
For $t=s-1$, by (\ref{realuss+1}), 
\begin{align*}
&|\frac{1}{N-\sum_{l=1}^{s-1}k_l}\sum_{n=1}^{N-\sum_{l=1}^{s-1}k_l}u^{(s-1)}_{k_1,...,k_{s-1},n}|^2  \\
\lesssim &\frac{1}{H_{s}} + \frac{1}{(N-\sum_{l=1}^{s-1}k_l)H_{s}^2} \sum_{k_{s}=1}^{H_{s}} (H_{s}-k_{s})(N-\sum_{l=1}^{s}k_l)  |\frac{\sum_{n=1}^{N-\sum_{l=1}^{s}k_l} u_{k_1,..., k_{s},n}^{(s)}}{(N-\sum_{l=1}^{s}k_l)}| \\
\lesssim &\frac{1}{H_{s}}=\frac{1}{H_{s-1}^2}.
\end{align*}
Then we procedd by reverse induction.  $\hfill{} \Box$
\end{proof}

At the final step we obtain
\begin{equation*}
|\frac{1}{N}\sum_{n=1}^N u^{(0)}_n|^2\lesssim \frac{1}{H_0^2}
\end{equation*}
Plugging it into (\ref{alphaETK}), we have
\begin{equation*}
D(\vec{Y}_1,...,\vec{Y}_N)\lesssim \frac{1}{H_0}+\sum_{0<|\vec{h}|\leq H_0} \frac{1}{r(\vec{h})}\frac{1}{H_0}\lesssim \frac{1}{H_0^{1-\epsilon}}\sim N^{-\frac{1-\epsilon}{(2^d-1)(\tau+\epsilon)}}.
\end{equation*}
$\hfill{} \Box$

\subsection{Liouvillean $\alpha$, Proof of Lemma \ref{skewwbdd}}
For $\alpha \notin DC(d)$, by property (\ref{NonDio}), we could find a subsequence $\{\frac{p_n}{q_n}\}$ of the continued fraction approximants of $\alpha$, so that $q_{n+1}> q_n^{d}$. In the following we will use $q$ instead of $q_n$ and $\tilde{q}$ instead of $q_{n+1}$ for simplicity.
Here we would like to show $D_q(\vec{Y}_1,...,\vec{Y}_q)\leq q^{-\delta}$ for some $\delta>0$.
Take 
\begin{equation}\label{H_jrational}
H_j\sim q^{\frac{2^j}{2^d}}\ \mathrm{for}\ 0\leq j \leq d-2\ \ \ \mathrm{and} \ \ \ H_{d-1}\sim q^{\frac{2^{d-1}(1+\epsilon)}{2^d}},
\end{equation}
where $\epsilon>0$ is small enough so that
\begin{equation}\label{<q}
\prod_{l=0}^{d-1}H_l=q^{\frac{2^d-1+2^{d-1}\epsilon}{2^d}}<q.
\end{equation}
Now by Lemma $\ref{ETK}$
\begin{align}\label{rationalETK}
D(\vec{Y}_1,..., \vec{Y}_q) 
\leq C_d (\frac{1}{H_0}+\sum_{0< |\vec{h}|\leq H_0} \frac{1}{r(\vec h)} \mid \frac{1}{q}\sum_{n=1}^q \exp\{2\pi i \sum_{j=1}^d (h_j\alpha+h_{j+1} y_1+...+h_d y_{d-j})\binom{n}{j}\}\mid)
\end{align}
Consider the following difference
\begin{align}\label{difference}
& \frac{1}{q}|\sum_{n=1}^q \exp\{2\pi i \sum_{j=1}^d (h_j\alpha+h_{j+1} y_1+...+h_d y_{d-j})\binom{n}{j}\}-\sum_{n=1}^q \exp\{2\pi i \sum_{j=1}^d (h_j\frac{p}{q}+h_{j+1} y_1+...+h_d y_{d-j})\binom{n}{j}\}|\\
\leq & \frac{1}{q}\sum_{n=1}^q |\exp\{2\pi i \sum_{j=1}^d h_j(\alpha-\frac{p}{q})\binom{n}{j}\}-1| \notag\\
\lesssim & \frac{1}{q} \sum_{n=1}^q \sum_{j=1}^d \binom{n}{j}H_0|\alpha-\frac{p}{q}| \notag\\
\lesssim &\frac{H_0}{q}, \notag
\end{align}
where in the last step we use (\ref{qnqn+1}), $|\alpha-\frac{p}{q}|\leq \frac{1}{q\tilde{q}}<\frac{1}{q^{d+1}}$.

Then combining (\ref{rationalETK}) with (\ref{difference}), we have
\begin{equation}\label{rationalu0}
D(\vec{Y}_1,..., \vec{Y}_q)  \lesssim C_d (\frac{1}{H_0}+\sum_{0< |\vec{h}|\leq H_0} \frac{1}{r(\vec h)} \mid \frac{1}{q}\sum_{n=1}^q u^{(0)}_n \mid)+\frac{H_0}{q},
\end{equation}
where $\tilde{u}^{(0)}_n= \exp\{2\pi i \sum_{j=1}^d (h_j\frac{p}{q}+h_{j+1} y_1+...+h_d y_{d-j})\binom{n}{j}\}$, that is $u_n^{(0)}$ as in (\ref{u0}) with $\alpha$ replaced with $\frac{p}{q}$. 
Thus with $\tilde{u}^{(s)}_{k_1,..., k_s, n}$ defined as in (\ref{us}) with $\alpha$ replaced with $\frac{p}{q}$,
similar to (\ref{alphad-2}) and (\ref{d-1estimatedetail}), we have
\begin{align}\label{rationald-2}
              &|\frac{1}{N-\sum_{l=1}^{d-2}k_l} \sum_{n=1}^{N-\sum_{l=1}^{d-2}k_l} \tilde{u}_{k_1,..., k_{d-2},n}^{(d-2)} |^2 \notag \\
\lesssim &\frac{1}{H_{d-1}} + \frac{1}{(N-\sum_{l=1}^{d-2}k_l)H_{d-1}} \sum_{k_{d-1}=1}^{H_{d-1}} |\sum_{n=1}^{N-\sum_{l=1}^{d-1}k_l} \tilde{u}_{k_1,..., k_{d-2},n}^{(d-2)} \overline{\tilde{u}_{k_1,..., k_{d-2},n+k_{d-1}}^{(d-2)}} |,
\end{align}
and
\begin{align}
  &|\sum_{n=1}^{q-\sum_{l=1}^{d-1}k_l}\tilde{u}_{k_1,...,k_{d-2},n}^{(d-2)}\overline{\tilde{u}_{k_1,...,k_{d-2},n+k_{d-1}}^{(d-2)}}| \notag\\
=&|\sum_{n=1}^{q-\sum_{l=1}^{d-1}k_l}\exp\{2\pi i h_d n \frac{p}{q} \prod_{l=1}^{d-1}k_l\}| \notag\\
\lesssim &\frac{1}{\|h_d \frac{p}{q} \prod_{l=1}^{d-1}k_l\|_{\R/\Z}}.    \label{rationald-1}
\end{align}
Since $|h_d|\leq H_0$, $1\leq k_i\leq H_i$ and ($\ref{<q}$), for any $1\leq k\leq H_{d-1}$ we have $\|k h_d \frac{p}{q}\prod_{l=1}^{d-2} k_l\|_{\R/\Z} \geq \frac{1}{q}$. Thus
\begin{equation}\label{1/q}
\sum_{k_{d-1}=1}^{H_{d-1}}\frac{1}{\|h_d \frac{p}{q} \prod_{l=1}^{d-1}k_l\|_{\R/\Z}}\lesssim \sum_{j=1}^{H_{d-1}}\frac{q}{j}\leq q\ln{H_{d-1}}.
\end{equation}
Then combining (\ref{rationald-2}), (\ref{rationald-1}) with (\ref{1/q}), we get
\begin{equation}
|\frac{1}{q-\sum_{l=1}^{d-2}k_l}\sum_{n=1}^{q-\sum_{l=1}^{d-2}k_l}\tilde{u}^{(d-2)}_{k_1,...,k_{d-2},n}|^2\lesssim \frac{1}{H_{d-1}}+\frac{q\ln{H_{d-1}}}{(q-\sum_{l=1}^{d-2}H_l)H_{d-1}}\lesssim \frac{1}{H^{\frac{1}{1+\epsilon}}_{d-1}}=\frac{1}{H_{d-2}^2}.
\end{equation}
By Lemma \ref{reverse},
\begin{align*}
|\frac{1}{q}\sum_{n=1}^q \tilde{u}_n^{(0)}|^2\lesssim \frac{1}{H_0}.
\end{align*}
Plugging it into (\ref{rationalu0}), we get
\begin{align*}
D(\vec{Y}_1,...,\vec{Y}_q)\lesssim \frac{1}{H_0}+\frac{(\log{H_0})^d}{H_0}+\frac{H_0}{q} \lesssim \frac{1}{q^{\frac{1-\epsilon}{2^d}}}.
\end{align*}

\section{Bounded remainder sets}
Most of the material covered in this section comes from \cite{GL}. We briefly discuss it here for completeness and readers' convenience.
From now on we restrict our attention to irrational rotation on $\T^d$. 
For a measurable set $U\subset \T^d$, consider the function $A_N(U,\vec{x})-N|U|:=A(U,  \{\vec{x}+n{\alpha}\}_{n=0}^{N-1})-N|U|=\sum_{n=0}^{N-1}\chi_{U}(\vec{x}+n\alpha)-N|U|$. 
We will say $U$ is a {\it bounded remainder set} ({\it BRS}) with respect to $\alpha$ if there exists a constant $C(U,\alpha)>0$ such that $|A_N(U,\vec{x})-N|U||\leq C(U,\alpha)$ for any $N$ and a.e. $\vec{x}\in \T^d$. 
We will call a measurable function $g$ on $\T^d$ a transfer function for $U$ if its characteristic function satisfies
\begin{align*}
\chi_{U}(\vec{x})-|U|=g(\vec{x})-g(\vec{x}-\alpha)\ \ \mathrm{a.e.}
\end{align*}
Obviously if $g$ is a transfer function for $U$, then its Fourier coefficients satisfy
\begin{align}\label{Fouriercoe}
\hat{g}(\vec{m})=\frac{\hat{\chi}_{U}(\vec{m})}{1-e^{-2\pi i \langle \vec{m}, \alpha\rangle }}, \ \ \vec{m}\neq 0.
\end{align}

\begin{prop}\label{setU}\cite{GL}
For a measurable set $U\subset \T^d$, the following are equivalent:
\begin{itemize}
\item $U$ is a bounded remainder set.
\item $U$ has a bounded transfer function $g$. 
\end{itemize}
\end{prop}

Theorems \ref{intervalBRS}, \ref{setBRS} and Corollary \ref{2dbound} are presented in \cite{GL} without explicit bounds on the transfer functions. We present the proofs in order to extract the needed estimates.

\begin{thm}\label{intervalBRS}
Any interval $I\subset \T$ of length $0<|q\alpha-p|<1$ is a {\it BRS} with respect to $\alpha$, furthermore its transfer function $g$ satisfies $\|g\|_{\infty} \leq |q|$.
\end{thm}
\begin{proof}
Without loss of generality, we consider an interval $I=[0,\kappa]$, where $\kappa=q\alpha-p>0$. Then
\begin{align*}
\chi_I(x)-|I|
&=-\{x\}+\{x-\kappa\}\\
&=-\{x\}+\{x-q\alpha\}\\
&=(-\{x\}-...-\{x-(q-1)\alpha\})+(\{x-\alpha\}+...+\{x-q\alpha\})\\
&=g(x)-g(x-\alpha),
\end{align*} where $g(x)=-\sum_{j=0}^{q-1}\{x-j\alpha\}$, $\|g\|_{\infty}\leq |q|$.
\end{proof} $\hfill{} \Box$

\begin{thm}\label{setBRS}
Let $\vec{v}=(v_1,v_2,...,v_d)=q\alpha-\vec{p} \in \Z\alpha+\Z^d$, $v\notin \Z^d$, and let $\Sigma\in \T^{d-1}$ be a {\it BRS} with respect to the vector $(\frac{v_1}{v_d}, \frac{v_2}{v_d},... \frac{v_{d-1}}{v_d})$ with transfer function $h$. Then the set 
\begin{align*}
U=U(\Sigma, \vec{v})=\{(\vec{x},0)+t\vec{v}: \vec{x}\in \Sigma, 0\leq t<1\},
\end{align*}
is a {\it BRS} with respect to $\alpha$, whose transfer function $g$ satisfies $\|g\|_{\infty}\leq |q|(\|h\|_{\infty}+1)$.
\end{thm}

\begin{proof}
Let $\vec{v}_0=(v_1,...,v_{d-1})$ be the vector in $\T^{d-1}$, which consists of the first $d-1$ entries of $\vec{v}$. First, we wish to find a bounded function $\tilde{g}$ on $\T^d$ satisfying the cohomological equation
\begin{align*}
\chi_U(\vec{x},y)-|U|=\tilde{g}(\vec{x},y)-\tilde{g}(\vec{x}-\vec{v}_0,y-v_d)\ \ \ \mathrm{for}\ \mathrm{a.e.}\ (\vec{x},y)\in \T^{d-1}\times \T.
\end{align*}
This means the Fourier coefficients satisfy the equation
\begin{align}
\hat{\tilde{g}}(\vec{m}, n)(1-e^{-2\pi i (\langle \vec{m}, \vec{v}_0 \rangle+nv_d)})=\int_{0}^{v_d}\int_{\Sigma+\frac{y}{v_d}\vec{v}_0}e^{-2\pi i \langle \vec{m}, \vec{x}+\frac{y}{v_d}\vec{v}_0 \rangle} \mathrm{d}\vec{x}\ e^{-2\pi i ny} \mathrm{d}y,\ \ (\vec{m}, n)\neq (\vec{0}, 0). 
\end{align}
Which implies
\begin{align}\label{coholoequ}
\hat{\tilde{g}}(\vec{m},n)=\frac{\hat{\chi}_{\Sigma}(\vec{m})}{2\pi i(\langle \vec{m}, \vec{v}_0 \rangle/{v_d}+n)}, \ \ (\vec{m}, n)\neq (\vec{0}, 0).
\end{align}   
We know $\Sigma$ is a $BRS$ with respect to $\vec{v}_0/{v_d}$, by (\ref{Fouriercoe}) its transfer function $h: \T^{d-1}\rightarrow\R$ satisfies
\begin{align*}
\hat{h}(\vec{m})=\frac{\hat{\chi}_{\Sigma}(\vec{m})}{1-e^{-2\pi i \langle \vec{m},\vec{v}_0\rangle /{v_d}}}, \ \ \vec{m}\neq 0.
\end{align*}
It is straightforward to check that the bounded function $\tilde{g}$ defined by
\begin{align*}
\tilde{g}(\vec{x},y)=h(\vec{x}-\frac{\vec{v}_0}{v_d}\{y\})-|\Sigma|\cdot \{y\},
\end{align*}
satisfies the coholomogical equation ($\ref{coholoequ}$).
Hence $\tilde{g}$ is a bounded transfer function for $U$ with respect  to $\vec{v}$.

Indeed, $\|\tilde{g}\|_{\infty}\leq \|h\|_{\infty}+1$. Since $\vec{v}=q\alpha-\vec{p}$, letting $g(\vec{x})=\tilde{g}(\vec{x})+\tilde{g}(\vec{x}-\alpha)+...+\tilde{g}(\vec{x}-(q-1)\alpha)$ we have that $U$ is a {\it BRS} with respect to $\alpha$ with bounded transfer function $g$ satisfying $\|g\|_{\infty}\leq |q|\|\tilde{g}\|_{\infty}\leq |q|(\|h\|_{\infty}+1)$.
\end{proof} $\hfill{} \Box$

The following corollary will be used several times in section 8.
\begin{cor}\label{2dbound}
Let $U\subset \T^2$ be the parallelogram spanned by two vectors $m(\alpha_1, \alpha_2)-(l_{1}, l_{2})$ and $(q\frac{m \alpha_1-l_1}{m \alpha_2-l_2}- p,\ 0)$, then $U$ is a {\it BRS} with respect to $(\alpha_1,\alpha_2)$ with transfer function $g$ satisfying $\|g\|_{\infty}\leq |m|(|q|+1)\leq 2|mq|$.
\end{cor}
\begin{proof}
In this case $\mathrm{v}=(v_1, v_2)=m(\alpha_1,\alpha_2)-(l_1,l_2)\in \Z\alpha+\Z^2$, $\Sigma=[0, q\frac{v_1}{v_2}-p]\times \{0\}$. We know the transfer function $h$ of $\Sigma$ with respect to $v_1/{v_2}$ satisfies $\|h\|_{\infty}\leq |q|$. Thus $\|g\|_{\infty}\leq |m|(|q|+1)\leq 2|mq|$.
\end{proof} $\hfill{} \Box$

\section{2-dimensional irrational rotation with weak diophantine frequencies}
In this section we deal with 2-dimensional weakly Diophantine frequencies. Our goal is to prove Lemma \ref{2dsteps}.

\subsection*{Proof of Lemma \ref{2dsteps}}
Assume $(\alpha_1, \alpha_2)\in WDC(c_0,\tau/4)$, for some $\tau>4$ and $c_0>0$. We divide the discussion into two parts.

First, we introduce the coprime Diophantine condition: 
\begin{align}\label{coprime}
PDC(\tau)=\cup_{c>0}PDC(c,\tau)=\cup_{c>0}\{(\alpha_1,\alpha_2)| 
 \|\langle \vec{h}, \mathbf{\alpha} \rangle\|_{\T} \geq \frac{c}{|\vec{h}|^{\tau}}\ &\mathrm{for}\ \mathrm{any}\ \gcd(h_1, h_2)=1\\ 
                                                                    & \mathrm{or}\ h_1h_2=0\ \mathrm{but}\ \vec{h}\neq \vec{0}\}.\notag
\end{align}
Obviously if $\alpha\in PDC(c, \tau)$,  both $\alpha_1$ and $\alpha_2$ belong to $DC(c, \tau)$.
\subsubsection*{Case A}
$(\alpha_1,\alpha_2)\in PDC(c_1,\tau)$ for some $c_1>0$. 

Let's take the best simultaneous approximation $\{(\frac{l_{1,n}}{m_n}, \frac{l_{2,n}}{m_n})\}$ of $(\alpha_1, \alpha_2)$. 
They feature the following property.
\begin{lemma}\label{2dnk}$\mathrm{(}$\cite{Lag}, Theorem 3.5$\mathrm{)}$
If $\{1,\alpha_1, \alpha_2\}$ is linearly independent over $\Q$, then there are infinitely many $n_k$ such that 
\begin{align*}
\left|
\begin{matrix}
&m_{n_k}\ \ \ \ \ l_{1,n_k}\ \ \ \ \ l_{2,n_k}\ \ \\
&m_{n_k+1}\ \ l_{1, n_k+1}\ \ l_{2, n_k+1}\ \ \\
&m_{n_k+2}\ \ l_{1, n_k+2}\ \ l_{2, n_k+2}\ \
\end{matrix}
\right|
\neq 0
\end{align*}
\end{lemma}

Now we take $r_k>0$ such that 
\begin{align}\label{rk}
m_{n_k}\leq \frac{4}{\pi} r_k^{-2}<m_{n_k+1}.
\end{align} 
By (\ref{bestsimqnqn+1}), the choice of $r_k$ guarantees that for $n\geq n_k$, 
\begin{align}\label{Amnball}
(m_n\alpha_1-l_{1,n}, m_n\alpha_2-l_{2,n}) \in B_{r_k}(0,0),
\end{align} 
where $B_r(x_1, x_2):=\{y=(y_1, y_2)\in \T^2: \|y_1-x_1\|_{\T}^2+\|y_2-x_2\|_{\T}^2<r_k^2\}$. 
Let $\{\frac{p_{n,s}}{q_{n,s}}\}_{s=1}^\infty$ be the continued fraction approximants of $\frac{m_n\alpha_1-l_{1,n}}{m_n\alpha_2-l_{2,n}}$. 
For each $n$ choose $s_n$ such that 
\begin{align}\label{sn}
q_{n,s_n}\leq r_k^{-1}< q_{n, s_n+1}.
\end{align} 
By (\ref{qnqn+1}), the choice of $s_n$ guarantees that 
\begin{align}\label{Asnball}
(q_{n,s_n}\frac{m_n\alpha_1-l_{1,n}}{m_n\alpha_2-l_{2,n}}-p_{n,s_n},0)\in B_{r_k}(0,0).
\end{align}
By (\ref{bestsimqnqn+1}) and (\ref{multiweakDio}) we have 
\begin{equation}\label{maxbound}
\frac{c_0}{m_n^{\tau/4}}\leq \max\{|m_n\alpha_1-l_{1,n}|, |m_n\alpha_2-l_{2,n}|\} \leq \frac{2}{\sqrt{\pi}\sqrt{m_{n+1}}},
\end{equation} 
by (\ref{rk}) we have $m_{n_k}\leq \frac{4}{\pi}r_k^{-2}$, thus 
\begin{align}\label{Amnk012}
\max{(m_{n_k}, m_{n_k+1}, m_{n_k+2})} \leq C_{c_0, \tau} r_k^{-\frac{\tau^2}{2}}.
\end{align}
\subsubsection*{Case A.1:} 
For some $n\in\{n_k, n_k+1, n_k+2\}$, we have $q_{n,s_n+1}\leq r_k^{-2\tau^4}$.

Let $U$ be the parallelogram spanned by the two vectors $m_n(\alpha_1,\alpha_2)-(l_{1,n}, l_{2,n})$ and $(q_{n,s_n}\frac{m_n\alpha_1-l_{1,n}}{m_n\alpha_2-l_{2,n}}-p_{n,s_n},\ 0)$. 
By (\ref{Amnball}) and (\ref{Asnball}), $U\subset B_{2r_k}(0, 0)$.
Corollary $\ref{2dbound}$ implies that $|\sum_{j=0}^{M-1}\chi_U(x+j\alpha_1, y+j\alpha_2)-M|U| |\leq 4|m_n q_{n,s_n}|$ for $a.e.\ (x,y)$. 
Thus as long as $M>\frac{4|m_n q_{n,s_n}|}{|U|}$, we should have $\cup_{j=0}^{M-1}U-(j\alpha_1, j\alpha_2)$ covers the whole $\T^2$ up to a measure zero set. 
Then 
\begin{align}\label{A1T2|U|}
\T^2 \subseteq \cup_{j=0}^{M-1} B_{2r_k}(-j\alpha_1, -j\alpha_2)\ \mathrm{for}\ M>\frac{4|m_n q_{n,s_n}|}{|U|}.
\end{align}
Now we want to estimate $|U|$. 
Since $ \alpha_2\in DC(c_1, \tau)$, by (\ref{Dioctau}) we have
\begin{align*}
|U|=|m_n\alpha_2-l_{2,n}|\cdot |q_{n,s_n}\frac{m_n\alpha_1-l_{1,n}}{m_n\alpha_2-l_{2,n}}-p_{n,s_n}|\geq \frac{c_1}{|m_n|^\tau}\frac{1}{2q_{n,s_n+1}}.
\end{align*}
Thus by (\ref{sn}) and (\ref{Amnk012}),
\begin{align*}
\frac{4|m_n|q_{n,s_n}}{|S|}
\leq \frac{8}{c_1} |m_n|^{1+\tau} q_{n,s_n}q_{n,s_n+1}
\leq C_{c_0, c_1, \tau} r_k^{-3\tau^4}.
\end{align*}
This means it takes $B_{2r_k}(0,0)$ at most $C_{\alpha_1,\alpha_2,\tau} r_k^{-3\tau^4}$ steps to cover the whole $\T^2$.
\subsubsection*{Case A.2} 
We will show now it is impossible to have $q_{n,s_n+1}>r_k^{-2\tau^4}$ for all $n\in\{n_k, n_k+1, n_k+2\}$. 
In this case by (\ref{qnqn+1}), (\ref{bestsimqnqn+1}) and (\ref{rk}), we have:
\begin{align}\label{defMn}
|q_{n,s_n}m_n\alpha_1-p_{n,s_n}m_n\alpha_2+M_n|=|m_n\alpha_2-l_{2,n}|\cdot |q_{n, s_n}\frac{m_n\alpha_1-l_{1,n}}{m_n\alpha_2-l_{2,n}}-p_{n,s_n}|< \frac{2}{\sqrt{\pi}\sqrt{|m_{n+1}|}q_{n, s_n}}<r_k^{2\tau^4+1}
\end{align}
where $M_n=p_{n,s_n}l_{2,n}-q_{n,s_n}l_{1,n}$. 

We have the following estimates on the upper bounds of $p_{n,s_n}$ and $M_n$. 
Combining (\ref{Dioctau}), (\ref{rk}), (\ref{sn}), (\ref{maxbound}) with (\ref{Amnk012}),
\begin{align}\label{pnsn}
|p_{n,s_n}|\leq q_{n,s_n}|\frac{m_n\alpha_1-l_{1,n}}{m_n\alpha_2-l_{2,n}}|+\frac{1}{q_{n, s_n+1}}\leq \frac{2q_{n, s_n}|m_n|^{\tau}}{c_1 \sqrt{\pi}\sqrt{|m_{n+1}|}}+r_k^{2\tau^4}\leq C_{c_0,c_1,\tau}r_k^{-\frac{\tau^3}{2}}.
\end{align}
By (\ref{defMn}), (\ref{rk}), (\ref{Amnk012}), (\ref{sn}) and (\ref{pnsn}),
\begin{align}\label{Mn}
|M_n|<|q_{n, s_n}m_n\alpha_1-p_{n, s_n}m_n\alpha_2|+r_k^{2\tau^4}\leq  C_{c_0,c_1,\tau} r_k^{-\tau^3}.
\end{align}
\subsubsection*{Case A.2.1} 
If $p_{n,s_n}=0$ for some $n\in\{n_k, n_k+1, n_k+2\}$, then by (\ref{qnqn+1}), (\ref{bestsimqnqn+1}) and (\ref{coprime}), (\ref{Dioctau}),(\ref{rk}), (\ref{Amnk012}), we have
\begin{align*}
r_k^{2\tau^4}>\frac{1}{q_{n,s_n+1}}\geq |q_{n,s_n}\frac{m_n\alpha_1-l_{1,n}}{m_n\alpha_2-l_{2,n}}|\geq \frac{c_1 \sqrt{\pi}\sqrt{|m_{n+1}|}}{2m_n^\tau}\geq C_{c_0,c_1,\tau}r_k^{\frac{\tau^3}{2}+1},
\end{align*}
which is a contradiction.
\subsubsection*{Case A.2.2} 
 If $M_n=0$ for some $n\in\{n_k, n_k+1, n_k+2\}$, then by (\ref{defMn}), (\ref{rk}), (\ref{pnsn}), and the fact that $(\alpha_1, \alpha_2)\in PDC(c_1, \tau)$, we have
\begin{align*}
r_k^{2\tau^4}>|m_n| |q_{n,s_n}\alpha_1-p_{n,s_n}\alpha_2|\geq \frac{c_1 |m_n|}{\max{(p_{n,s_n}, q_{n,s_n})}^\tau}\geq C_{c_0,c_1,\tau}r_k^{\frac{\tau^4}{2}},
\end{align*} again a contradiction.
\subsubsection*{Case A.2.3}
If $p_{n,s_n}\neq 0$ and $M_n\neq 0$ for any $n\in \{n_k, n_k+1, n_k+2\}$, then for any $i,j\in \{n_k, n_k+1, n_k+2\}$, we have:
\begin{align}\label{MiMj}
&|(q_{i, s_i}m_iM_{j}-q_{j,s_{j}}m_{j}M_i)\alpha_1-(p_{i,s_i}m_iM_{j}-p_{j,s_{j}}m_{j}M_i)\alpha_2|\\
\leq &|(q_{i, s_i}m_i\alpha_1-p_{i, s_i}m_i\alpha_2+M_i)M_j|+|(q_{j, s_j}m_j\alpha_1-p_{j, s_j}m_j\alpha_2+M_j)M_i| \notag\\
<&(|M_i|+|M_{j}|) r_k^{2\tau^4}.\notag
\end{align}
\subsubsection*{Case A.2.3.1} 
 $(q_{i,s_i} m_i M_j- q_{j,s_j} m_j M_i,\ p_{i, s_i} m_i M_j- p_{j, s_j} m_j M_i)\neq (0,0)$ for some $i,j\in\{n_k, n_k+1,n_k+2\}$.

In this case let $h=\gcd (q_{i,s_i} m_i M_j- q_{j,s_j} m_j M_i,\ p_{i, s_i} m_i M_j- p_{j, s_j} m_j M_i)$ be the greatest common divisor of the two numbers if they are both nonzero, and $h=1$ otherwise. Then by (\ref{MiMj}),
\begin{align*}
|\frac{q_{i,s_i} m_i M_j- q_{j,s_j} m_j M_i}{h}\alpha_1-\frac{p_{i, s_i} m_i M_j- p_{j, s_j} m_j M_i}{h}\alpha_2|<\frac{|M_i|+|M_j|}{h}r_k^{2\tau^4}.
\end{align*}
However on one hand by (\ref{Mn}),
\begin{align*}
\frac{|M_i|+|M_j|}{h}r_k^{2\tau^4}\leq (|M_i|+|M_j|)r_k^{2\tau^4}\leq C_{c_0,c_1,\tau}r_k^{2\tau^4-\tau^3}.
\end{align*}
On the other hand, by the fact that $(\alpha_1, \alpha_2)\in PDC(c_1, \tau)$ and (\ref{rk}), (\ref{Amnk012}), (\ref{pnsn}), (\ref{Mn}),
\begin{align*}
        &|\frac{q_{i,s_i} m_i M_j- q_{j,s_j} m_j M_i}{h}\alpha_1-\frac{p_{i, s_i} m_i M_j- p_{j, s_j} m_j M_i}{h}\alpha_2|\\
\geq & \frac{c_1 h^{\tau}}{|(q_{i,s_i} m_i M_j- q_{j,s_j} m_j M_i, p_{i, s_i} m_i M_j- p_{j, s_j} m_j M_i)|^\tau}\\
\geq & C_{c_0,c_1,\tau}r_k^{\frac{7}{4}\tau^4},
\end{align*}  a contradiction.
\subsubsection*{Case A.2.3.2} For any $i,j\in \{n_k, n_k+1, n_k+2\}$
\begin{align*}
&q_{i,s_i}m_i M_j=q_{j,s_j}m_j M_i\\
&p_{i,s_i}m_i M_j=p_{j,s_j}m_j M_i.
\end{align*}
Then for $n=n_k$,
\begin{align*}
\frac{p_{n,s_n}}{q_{n,s_n}}=\frac{p_{n+1,s_{n+1}}}{q_{n+1,s_{n+1}}}=\frac{p_{n+2,s_{n+2}}}{q_{n+2,s_{n+2}}}.
\end{align*}
Hence we can let $p=p_{n,s_n}=p_{n+1,s_{n+1}}=p_{n+2,s_{n+2}}$ and $q=q_{n,s_n}=q_{n+1,s_{n+1}}=q_{n+2,s_{n+2}}$.
Then we would have (after plugging in $M_n=q l_{1,n}-p l_{2,n}$)
\begin{equation}\label{10.1}
q(m_n l_{1,n+1}-m_{n+1} l_{1,n})=p(m_n l_{2,n+1}-m_{n+1} l_{2,n})
\end{equation}
\begin{equation}\label{10.2}
q(m_n l_{1,n+2}-m_{n+2} l_{1,n})=p(m_n l_{2,n+2}-m_{n+2} l_{2,n})
\end{equation}
\begin{equation}\label{10.3}
q(m_{n+1} l_{1,n+2}-m_{n+2} l_{1,n+1})=p(m_{n+1} l_{2,n+2}-m_{n+2} l_{2,n+1})
\end{equation}
Then considering
$(\ref{10.1})\cdot (-l_{1,n+2})+(\ref{10.2})\cdot l_{1, n+1}+ (\ref{10.3})\cdot (-l_{1,n})$, we get 
\begin{align*}
p \cdot
\left|
\begin{matrix}
&m_{n_k}\ \ \ \ \ l_{1,n_k}\ \ \ \ \ l_{2,n_k}\ \ \\
&m_{n_k+1}\ \ l_{1, n_k+1}\ \ l_{2, n_k+1}\ \ \\
&m_{n_k+2}\ \ l_{1, n_k+2}\ \ l_{2, n_k+2}\ \
\end{matrix}
\right|
=q \cdot 0=0,
\end{align*} a contradiction with the choice of $n_k$.

\subsubsection*{Case B}
 $(\alpha_1, \alpha_2)\notin PDC(\tau)$. By the definition of $PDC(\tau)$, the sequence $\vec{h}_n=(h_{1,n}, h_{2,n})$ for which (\ref{coprime}) fails has to satisfy either $\gcd{(h_{1,n}, h_{2,n})}=1$ (Case B.1) or $h_{1,n}h_{2,n}=0$ (Case B.2). 
\subsubsection*{Case B.1}
We can find a sequence $\{n_j\}$, such that $|\vec{h}_{n_j}|=\max{(|h_{1,n_j}|, |h_{2,n_j}|)} \rightarrow \infty$ as $j\rightarrow \infty$, $\gcd{(h_{1,n_j}, h_{2,n_j})}=1$ and $\|h_{1,n_j}\alpha_1+h_{2,n_j}\alpha_2\|_{\T}<\frac{1}{|\vec{h}_{n_j}|^\tau}$. 

Without loss of generality, we can assume $|h_{1,n_j}|=|\vec{h}_{n_j}|$. In this case we can take $r_{n_j}=\frac{1}{|h_{1,n_j}|}$. For simplicity we will denote $n_j$ by $n$. 

Now that $\|h_{1,n}\alpha_1+h_{2,n}\alpha_2\|_{\T}<\frac{1}{|h_{1,n}|^\tau}$, we can find $l_{1,n}, l_{2,n}\in\Z$ such that $|h_{1,n}(\alpha_1-l_{1,n})+h_{2,n}(\alpha_2-l_{2,n})|<\frac{1}{|h_{1,n}|^\tau}$. Since replacing $(\alpha_1, \alpha_2)$ with $(\alpha_1+l_{1,n}, \alpha_2+l_{2,n})$ would not change anything, we will assume $|h_{1,n}\alpha_1+h_{2,n}\alpha_2|<\frac{1}{|h_{1,n}|^\tau}$. Then 
\begin{align}\label{alpha2/1}
|\frac{\alpha_2}{\alpha_1}-(-\frac{h_{1,n}}{h_{2,n}})|<\frac{1}{|h_{1,n}|^\tau \alpha_1}.
\end{align}
We consider the following two lines on $\T^2$: 
\begin{align*}
l_1(t)=(\{t\},\{\frac{\alpha_2}{\alpha_1}t\})\ \ \mathrm{and}\ \ l_2(t)=(\{t\},\{-\frac{h_{1,n}}{h_{2,n}}t\}).
\end{align*}
These two lines are close to each other in the sense that for $|t|\leq |h_{1,n}|^{3\tau/4}$, by (\ref{alpha2/1}),
\begin{align*}
\|\{\frac{\alpha_2}{\alpha_1}t\}-\{-\frac{h_{1,n}}{h_{2,n}}t\}\|_{\T} \leq |\frac{\alpha_2}{\alpha_1}t+\frac{h_{1,n}}{h_{2,n}}t|\leq \frac{|t|}{|h_{1,n}|^\tau \alpha_1}\leq \frac{1}{|h_{1,n}|^{\tau/4}\alpha_1}.
\end{align*}
The graph of $l_2(t)$ is the hypotenuse of a right triangle with two legs of lengths $|h_{1,n}|$ and $|h_{2,n}|$ (mod $\Z^2$). 
We consider the orbit of $(\alpha_1, -\frac{h_{1,n}}{h_{2,n}}\alpha_1)$ under the rotation $(\alpha_1, -\frac{h_{1,n}}{h_{2,n}}\alpha_1)$. These points lie on $l_2(t)$. 
Under this rotation the point moves a distance $\frac{\sqrt{h_{1,n}^2+h_{2,n}^2}}{|h_{2,n}|}\alpha_1$ at each step by a big interval with length $\sqrt{h_{1,n}^2+h_{2,n}^2}$. 
Let $\{\frac{p_m}{q_m}\}_{m=1}^\infty$ be the continued fraction approximants of $\frac{\alpha_1}{h_{2,n}}$. 
Choose $m$ such that 
\begin{align}\label{primeqm}
q_{m-1}\leq |h_{1,n}|\sqrt{h_{1,n}^2+h_{2,n}^2}<q_m.
\end{align} 
Then it would take a point on $\T$ at most $q_m+q_{m-1}$ steps (under the $\frac{\alpha_1}{h_{2,n}}$ -rotation) to enter each interval of length $\frac{1}{|h_{1,n}|\sqrt{h_{1,n}^2+h_{2,n}^2}}$ on $\T$  (e.g. \cite{JL}), 
which means it would take a point on $l_2(t)$ at most $q_m+q_{m-1}-1$ steps (under the $\frac{\sqrt{h_{1,n}^2+h_{2,n}^2}\alpha_1}{|h_{2,n}|}$-rotation) to enter each interval of length $\frac{1}{|h_{1,n}|}=r_n$ on the graph of $l_2(t)$.
Moreover, it is easy to see that the distance from any $x\in \T^2$ to $l_2(t)$ is bounded by $\frac{1}{\sqrt{h_{1, n}^2+h_{2, n}^2}}<r_n$. Thus
\begin{align}\label{T2ballh1h2}
\T^2\subseteq \cup_{k=0}^{q_m+q_{m-1}}B_{2r_n}(k\alpha_1,-\frac{h_{1,n}}{h_{2,n}}k\alpha_1).
\end{align}
By (\ref{qnqn+1}) and (\ref{alpha2/1}),
\begin{align*}
|p_{m-1}+q_{m-1}\frac{\alpha_2}{h_{1,n}}|=|p_{m-1}-q_{m-1}\frac{\alpha_1}{h_{2,n}}+q_{m-1}(\frac{\alpha_1}{h_{2,n}}+\frac{\alpha_2}{h_{1,n}})|\leq \frac{1}{q_m}+\frac{q_{m-1}}{|h_{1,n}|^{\tau-1}}.
\end{align*}
This implies, by (\ref{qnqn+1}) and (\ref{primeqm}),
\begin{align*}
&\|q_{m-1}\alpha_1\|_{\T}\leq |q_{m-1}\alpha_1-h_{2,n}p_{m-1}|\leq  \frac{|h_{2,n}|}{q_m},\\
&\|q_{m-1}\alpha_2\|_{\T}\leq \frac{|h_{1,n}|}{q_m}+\frac{2}{|h_{1,n}|^{\tau-4}}.
\end{align*}
Then by the fact that $\alpha\in WDC(c_0, \frac{\tau}{4})$ and (\ref{primeqm}),
\begin{align*}
\max\{\frac{|h_{2,n}|}{q_m}, \frac{|h_{1,n}|}{q_m}+\frac{2}{|h_{1,n}|^{\tau-4}}\}\geq \max{(\|q_{m-1}\alpha_1\|_{\T}, \|q_{m-1}\alpha_2\|_{\T})}\geq \frac{c_0}{q_{m-1}^{\tau/4}}\geq \frac{c_0}{2^{\frac{\tau}{4}}|h_{1,n}|^{\tau/2}}.
\end{align*}
This implies 
\begin{align}\label{T2qmqm-1}
q_m+q_{m-1}< 2q_m\leq \frac{2^{\frac{\tau}{4}+2}}{c_0}|h_{1,n}|^{\tau/2+1}.
\end{align} 
Since $0\leq k\leq \frac{2^{\frac{\tau}{4}+2}}{c_0} |h_{1,n}|^{\tau/2+1}<r_n^{-\frac{3\tau}{4}}$, by (\ref{alpha2/1}) the points $(k\alpha_1, k\alpha_2)$ and $(k\alpha_1, -\frac{h_{1,n}}{h_{2,n}}k\alpha_1)$ differ at most by $r_n^{\frac{\tau}{4}}$, so we obtain using (\ref{T2ballh1h2}) and (\ref{T2qmqm-1}),
\begin{align*}
\T^2 \subseteq \cup_{k=0}^{r_n^{-3\tau/4}}B_{3r_n}(k\alpha_1, k\alpha_2).
\end{align*}
\subsubsection*{Case B.2}
We can find a sequence $\{n_j\}$ such that $h_{2, n_j}\equiv 0$ and $|h_{1,n_j}|\rightarrow \infty$ such that 
\begin{align}\label{B2as}
\|h_{1,n_j}\alpha_1\|_{\T}<\frac{1}{|h_{1, n_j}|^\tau}.
\end{align} 
For simplicity we will replace $n_j$ with $n$. 
We can find $M_n$ such that $|h_{1,n}\alpha_1-M_n|<\frac{1}{|h_{1,n}|^\tau}$. 
Let $d_n=\gcd (h_{1,n}, M_n)$ be the greatest common divisor. Let $\tilde{h}_{1,n}=\frac{h_{1,n}}{d_n}$ and $\tilde{M}_n=\frac{M_n}{d_n}$. 
We have 
\begin{align}\label{y}
|\alpha_1-\frac{\tilde{M}_n}{\tilde{h}_{1,n}}|<\frac{1}{|h_{1,n}|^{\tau+1}}\rightarrow 0.
\end{align}
If $\tilde{h}_{1,n}$ is bounded in $n$, then $\alpha_1$ can be approximated arbitrarily closely by rationals with bounded denominators, which is impossible. 
Thus $|\tilde{h}_{1,n}|\rightarrow \infty$. 
Now take radius $r_n=\frac{1}{|\tilde{h}_{1,n}|}$. 
For each $0\leq i\leq \tilde{h}_{1,n}-1$ consider $\{(i\alpha_1+k\tilde{h}_{1,n}\alpha_1,\ i\alpha_2+k\tilde{h}_{1,n}\alpha_2)\}_{k=0}^\infty$. Let $\{\frac{p_m}{q_m}\}_{m=1}^\infty$ be the continued fraction approximants of $\tilde{h}_{1,n}\alpha_2$. Choose $m$ such that 
\begin{align}\label{w}
q_{m-1}\leq |\tilde{h}_{1,n}|=r_n^{-1}<q_m.
\end{align} 
Then it takes any point on $\T$ at most $q_m+q_{m-1}-1$ steps (under the $\tilde{h}_{1,n}\alpha_2-$rotation) to enter each interval of length $r_n$ \cite{JL}.
By (\ref{qnqn+1}), 
\begin{align}\label{x}
|p_{m-1}-q_{m-1}\tilde{h}_{1,n}\alpha_2|\leq \frac{1}{q_m}.
\end{align}
By (\ref{B2as}), (\ref{w}) and since $\tau>4$, we have $\|q_{m-1}\tilde{h}_{1,n}\alpha_1\|\leq \frac{q_{m-1}}{|\tilde{h}_{1,n}|^{\tau}}<\frac{c_0}{(q_{m-1}|\tilde{h}_{1,n}|)^{\tau/4}}$. 
By the fact that $\alpha\in WDC(c_0, \frac{\tau}{4})$, $\|q_{m-1}\tilde{h}_{1,n}\alpha_2\|\geq \frac{c_0}{(q_{m-1}|\tilde{h}_{1,n}|)^{\tau/4}}$. 
By (\ref{x}) and (\ref{w}), we have 
\begin{align}\label{z}
q_m\leq \frac{1}{c_0}|\tilde{h}_{1,n}|^{\frac{\tau}{2}}.
\end{align}
Now for $0\leq k\leq q_m+q_{m-1}-1$, by (\ref{y}), (\ref{B2as}) and (\ref{z}), $\|i\alpha_1+k\tilde{h}_{1,n}\alpha_1-\frac{i\tilde{M_n}}{\tilde{h}_{1,n}}\|_{\T} \leq \frac{C}{|\tilde{h}_{1,n}|^\frac{\tau}{2}}=Cr_n^{\frac{\tau}{2}}$. 
Since $\gcd{(\tilde{h}_{1,n}, \tilde{M}_n})=1$, any interval of length $r_n=\frac{1}{|\tilde{h}_{1,n}|}$ contains $\frac{i\tilde{M}_n}{\tilde{h}_{1,n}}$ for some $0\leq i\leq \tilde{h}_{1,n}-1$.
Thus 
\begin{align*}
\T^2\subseteq \cup_{k=0}^{(q_m+q_{m-1})|\tilde{h}_{1,n}|}B_{r_n}(k\alpha_1, k\alpha_2).
\end{align*}
By (\ref{z}), $(q_m+q_{m-1})|\tilde{h}_{1,n}|\leq r_n^{-\tau}$, so we have
\begin{align}
\T^2\subseteq \cup_{k=0}^{r_n^{-\tau}}B_{r_n}(k\alpha_1, k\alpha_2).
\end{align}
$\hfill{} \Box$.

\section*{\\Appendix A}
\subsection*{Proof of Lemma \ref{toralsbdd}}


We include the proof here for completeness.

For sufficiently small $\epsilon>0$, fix an integer $H_0\sim N^{1/(d(\tau-1)+1+d\epsilon)}$, define $g(n)=\frac{1}{n(n+1)}$ for $1\leq n<H_0$ and $g(H_0)=\frac{1}{H_0}$. For $(n_1,...,n_d)\in \Z^d$ with $1\leq n_i\leq H_0$, define $f(n_1,...,n_d)=\prod_{i=1}^d g(n_i)$.
By Lemma \ref{ETK}, we have
\begin{align*}
D_N(\theta)
&\leq C_d(\frac{1}{H_0}+\sum_{0<|h|\leq H_0}\frac{1}{r(\vec{h})}|\frac{1}{N}\sum_{n=1}^N e^{2\pi i \langle \vec{h}, {\alpha}\rangle n}|)\\
&\leq \tilde{C}_d(\frac{1}{H_0}+\frac{1}{N}\sum_{0<|h|\leq H_0} \frac{1}{r(\vec{h})} \frac{1}{\|\langle \vec{h}, {\alpha}\rangle \|_{\T}})\\
&=\tilde{C}_d(\frac{1}{H_0}+\frac{1}{N}\sum_{n_1,...,n_d=1}^{H_0} f(n_1,...,n_d)\sum_{\vec{h}=(h_1,...,h_d)\neq \vec{0}, |h_j|\leq n_j} \frac{1}{\|\langle \vec{h}, {\alpha} \rangle \|_{\T}})\\
&\leq \tilde{C}_d(\frac{1}{H_0}+\frac{1}{N}\sum_{n_1,...,n_d=1}^{H_0} f(n_1,...,n_d) \sum_{j=1}^{3^d r(\vec{n})}\frac{r(\vec{n})^\tau}{j})\\
&\leq \tilde{C}_d(\frac{1}{H_0}+\frac{1}{N}\sum_{n_1,...,n_d=1}^{H_0} f(n_1,...,n_d) r(\vec{n})^\tau \log{r(\vec{n})})\\
&\leq \tilde{C}_d(\frac{1}{H_0}+\frac{{H_0}^{d(\tau-1+\epsilon)}}{N})\\
&\lesssim N^{-1/(d(\tau-1)+1+d\epsilon)}.
\end{align*}
$\hfill{} \Box$

\section*{Acknowledgement}
R. H. would like to thank Anton Gorodetski for valuable discussions. 
This research was partially supported by the NSF DMS-1401204.  S. J. would like to thank the support of the Simons Foundation where she was a Fellow in 2014-15.
We are also grateful to the Isaac
    Newton Institute for Mathematical Sciences, Cambridge, for its
    hospitality, supported by EPSRC Grant Number EP/K032208/1, during the programme Periodic and Ergodic Spectral
    Problems where this work was started.

\bibliographystyle{amsplain}

\end{document}